\documentclass[11pt]{amsart}
\usepackage{geometry}                
\usepackage{amssymb}

\usepackage{tikz}
\usepackage{rank-2-roots}
\usetikzlibrary{calc,decorations.markings,matrix,arrows,math}

\usepackage[numbers,sort]{natbib} 

\usepackage{amsmath}
\usepackage{graphicx}
\usepackage{ytableau}
\usepackage{caption}

\usepackage{ytableau}
\usepackage[enableskew]{youngtab}

\usepackage{subcaption} 

\newtheorem*{Thm*}{Theorem}
\newtheorem{Thm}{Theorem}
\newtheorem{Prop}[Thm]{Proposition}
\newtheorem{Lem}[Thm]{Lemma}

\theoremstyle{definition}
\newtheorem{Def}[Thm]{Definition}

\theoremstyle{remark}

\newtheorem{Ex}[Thm]{Example}

\theoremstyle{definition}

\newcommand\scalemath[2]{\scalebox{#1}{\mbox{\ensuremath{\displaystyle #2}}}}

\DeclareMathOperator{\Hom}{Hom}
\DeclareMathOperator{\minconv}{\bf minconv}
\DeclareMathOperator{\maxconv}{\bf maxconv}
\DeclareMathOperator{\conv}{\bf conv}
\DeclareMathOperator{\Poly}{Poly}
\DeclareMathOperator{\dif}{dif}
\DeclareMathOperator{\vspan}{span}
\DeclareMathOperator{\sort}{sort}
\DeclareMathOperator{\Inv}{Inv}

\tikzset{arrowto/.style={decoration={markings, mark=at position .5 with {\arrow{>}}}, postaction={decorate}}}
\tikzset{arrowfrom/.style={decoration={markings, mark=at position .5 with {\arrow{<}}}, postaction={decorate}}}

\title{Non-elliptic Webs and Convex Sets in the Affine Building}
\author{Tair Akhmejanov}

\begin{document}

\begin{abstract}
We describe the $\mathfrak sl_3$ non-elliptic webs in terms of convex sets in the affine building. Kuperberg defined the non-elliptic web basis in his work on rank-$2$ spider categories. Fontaine, Kamnitzer, Kuperberg showed that the $\mathfrak sl_3$ non-elliptic webs are dual to CAT(0) triangulated diskoids in the affine building. We show that each such triangulated diskoid is the intersection of the min-convex and max-convex hulls of a generic polygon in the building. Choosing a generic polygon from each of the components of the Satake fiber produces (the duals of) the non-elliptic web basis. The convex hulls in the affine building were first introduced by Faltings and are related to tropical convexity, as discussed in work by Joswig, Sturmfels, Yu and by Zhang. 
\end{abstract}

\maketitle
\tableofcontents

\section{Introduction}

The representation theory of a Lie algebra or quantum group forms a pivotal tensor category that can be studied combinatorially via a diagrammatic presentation by generating morphisms and relations. One can restrict to the full subcategory with objects consisting of tensor products of the fundamental representations of $\mathfrak g$, the entire representation category being recovered from this subcategory by idempotent completion. The diagrammatic category corresponding to this full subcategory is called the \emph{spider category of $\mathfrak g$}. The morphisms are linear combinations of \emph{webs}, which are directed, planar graphs with trivalent interior vertices and univalent boundary vertices. They satisfy certain relations, which were completely specified in the $\mathfrak sl_m$ case in \cite{CKM}.

In \cite{Kup}, Kuperberg introduced the spider categories for rank-2 Lie algebras and gave presentations for these categories. Spider categories can be used to describe the invariant spaces $\Hom_{\mathfrak g}(\mathbb C,V_1\otimes\cdots\otimes V_n)\cong\Inv(V_1\otimes\cdots\otimes V_n)$ where each $V_i$ is a fundamental representation of $\mathfrak g$. Furthermore, each hom-space is equivalent to an invariant space due to the pivotal structure, and the spider categories were originally defined in \cite{Kup} entirely in terms of invariant spaces. From this perspective webs are thought of as embedded in a disk, and linear combinations of webs with boundary corresponding to $V_1,\ldots,V_n$ represent invariant vectors. In this rank-2 setting, Kuperberg specified a basis for each invariant space consisting of the \emph{non-elliptic webs} (see \S \ref{sec:webs}). They are the webs such that each interior face has at least 6 sides. We will be concerned with the non-elliptic webs in the $\mathfrak g=\mathfrak sl_3$ case.

Fontaine, Kamnitzer, Kuperberg \cite{FKK} showed that the $\mathfrak sl_3$ non-elliptic webs are dual to CAT(0) triangulated diskoids in the affine building. That is, since webs are trivalent graphs with boundary vertices on a disk, each web is dual to a triangulated polygon with edges labelled by the two fundamental weights of $\mathfrak sl_3$. If the web has multiple connected components, then certain boundary vertices of the polygon are identified, resulting in a triangulated diskoid. The non-elliptic property corresponds to the CAT(0) property of the triangulated diskoid, which says that each interior vertex has degree at least 6 (see Figure~\ref{fig:dual examples}).

This paper gives a description of these CAT(0) triangulated dual diskoids in terms of convex sets in the affine building. To state the main theorem we briefly introduce the $A_2$ affine building and define the convex sets in the building (precise definitions are given \S\ref{sec:building}).

The $A_2$ affine building $\Delta_2$ is an infinite simplicial complex of dimension $2$. A vertex in $\Delta_2$ is represented by a $\mathbb C[[t]]$-lattice in $\mathbb C((t))^3$ (more precisely, a homothety class of lattices where two lattices $L,L'$ are in the same class if $L'=t^kL$ for some integer $k$). Let $[L]$ denote the class of the lattice $L$. Two vertices of $\Delta_2$ are connected by an edge if there are lattice representatives $L,L'$ for their lattice classes such that $tL\subset L'\subset L$. Three vertices form a simplex if each pair forms an edge. There is an $SL_3$-dominant-weight-valued metric on the vertices of the building in the sense of Kapovich, Leeb, Milson \cite{KLM}. Let $\omega_1$ and $\omega_2$ be the two fundamental weights of $SL_3$ and let $V_{\omega_i}$ denote the corresponding fundamental representation. Note that two vertices $x_1,x_2$ in $\Delta_2$ form an edge if and only if $d(x_1,x_2)=\omega_1$ or $\omega_2$.

The following notion of convexity in the affine building is originally due to Faltings \cite{Fal}. Note that the intersection of two lattices is again a lattice. A set of lattice classes is \emph{min-convex} if it is closed under taking intersections. Then define the convex hull $\minconv([L_1],\ldots,[L_n])$ to be the smallest min-convex set containing the lattice classes $[L_1],\ldots,[L_n]$. There is a dual notion of max-convexity where intersection of lattices is replaced by taking their sum. Define $\maxconv([L_1],\ldots,[L_n])$ analogously. Joswig, Sturmfels, Yu \cite{JSY} gave an algorithm for computing these convex hulls in terms of tropical convexity, which was improved upon by Zhang \cite{Zha}. 

Finally, define the following notation for the intersection of these two convex hulls,
\begin{align*}
\conv([L_1],\ldots,[L_n])=\minconv([L_1],\ldots,[L_n])\cap \maxconv([L_1],\ldots,[L_n]).
\end{align*}

Let $\vec\lambda=(\lambda^1,\ldots,\lambda^n)$ be a sequence of fundamental weights. We will be interested in the space of polygon configurations $\Poly(\vec\lambda)$ consisting of $n$ vertices $(x_1,\ldots,x_n)$ in $\Delta_2$ such that $d(x_i,x_{i+1})=\lambda^i, d(x_n,x_1)=\lambda^n$ and $x_1$ is the class of the standard lattice $\mathbb C[[t]]^3$. This forms an algebraic variety, and by the geometric Satake correspondence, which we briefly recall below in this introduction (and in \S\ref{sec:grassmannian}), the number of components of the polygon space is equal to the dimension of $\Inv(\vec\lambda)=\Inv(V_{\lambda^1}\otimes\cdots\otimes V_{\lambda^n})$.

Let $P=(x_1,\ldots,x_n)$ be a polygon representing a point of $\Poly(\vec\lambda)$. Although $\conv(P)$ denotes a set of vertices in the affine building, we will interpret it as the induced subcomplex on those vertices. The main theorem is Theorem \ref{thm:main}, reproduced here.

\begin{Thm*}
Let $P$ be a generic point of a component of $\Poly(\vec\lambda)$, thought of as a polygon in the affine building $\Delta_2$. Then $\conv(P)$ is a CAT(0) triangulated diskoid, that is, the dual of $\conv(P)$ is a non-elliptic web. As $P$ ranges over the components of $\Poly(\vec\lambda)$, the duals of $\conv(P)$ form the non-elliptic web basis for the invariant space $\Inv(\vec\lambda)$.
\end{Thm*}

The main combinatorial tool used in the proof is the cylindrical growth diagram of \cite{Spe, Whi, Akh}. The definition is given in \S\ref{sec:grassmannian} together with the result from \cite{FK,Akh} that the components of $\Poly(\vec\lambda)$ can be described combinatorially via growth diagrams. A polygon $P$ is said to be generic if its pairwise distances $d(x_i,x_j)$ form a growth diagram.

Intuitively, the set of vertices $\minconv(P)$ contains certain combinatorial geodesics between pairs of vertices of $P$. The subcomplex it induces contains the desired CAT(0) triangulation of $P$, but may also contain extra simplices. The same holds for $\maxconv(P)$. Taking the intersection of these two sets of vertices induces a subcomplex that is exactly equal to the CAT(0) triangulation of $P$. The reader may look ahead to the figures in Example \ref{ex:octagon}.

The theorem can also be interpreted as giving a bijective map from components of $\Poly(\vec\lambda)$ to non-elliptic webs: choose a generic point in the component, compute $\conv(P)$, and take the dual of this triangulated diskoid. However, we do not explore computational procedures for computing these convex hulls in the present paper, but refer the reader to \cite{JSY} and \cite{Zha}. Such a bijection was already given in \cite[Thm 1.4]{FKK}, or more completely, we have the following objects in bijective correspondence for a fixed sequence of fundamental weights $\vec\lambda$.
\begin{itemize}
\item
non-elliptic webs with boundary $\vec\lambda$
\item
CAT(0) triangulated dual-diskoids with boundary $\vec\lambda$
\item
components of $\Poly(\vec\lambda)$
\end{itemize}

Fontaine, Kamnitzer, Kuperberg also studied the relation between the non-elliptic web basis and the Satake basis given by the components of $\Poly(\vec\lambda)$. They showed that for each invariant space the transition matrix between the bases is unitriangular. To briefly recall, the geometric Satake correspondence of Lusztig \cite{Lus}, Ginzburg \cite{Gin}, Beilinson--Drinfeld \cite{BD}, and Mirkovi\'c-Vilonen \cite{MV} describes the representation theory of $G$ in terms of the affine Grassmannian $Gr_{G^L}$ of the Langlands dual group $G^L$ (see \S\ref{sec:grassmannian} for precise definitions). In the present case, $G=SL_3$, $G^L=PGL_3$, and $Gr_{G^L}=PGL_3(\mathbb C((t)))/PGL_3(\mathbb C[[t]])$. Note that, as a set, $PGL_3(\mathbb C((t)))/PGL_3(\mathbb C[[t]])$ is the set of homothety classes of lattices, equivalently the vertices of $\Delta_2$. Under this correspondence, the invariant space $\Inv(V_{\lambda^1}\otimes \cdots\otimes V_{\lambda^n})$ is isomorphic to $H_{top}(\Poly(\vec\lambda))$, the top Borel--Moore homology of the polygon configuration space. The polygon space, also known as the Satake fiber, is a reducible variety whose top components in Borel--Moore homology form the Satake basis for $\Inv(\vec\lambda)$. 

The non-elliptic webs have the nice property that they are minimal with respect to the number of interior vertices. In higher rank, there is no known method for generating basis webs that is preserved under rotation and the resulting webs minimal under some natural ordering. However, there are known ways for generating bases, as given in \cite{Fon} and \cite{CKM} (see also \cite{Wes} and \cite{Hag}). A better understanding of the connection to convexity may shed light on good sets of bases in higher rank, which was one of the motivations for establishing this interpretation in rank $2$.

Part of the content of the main theorem is that $\conv(P)$ contains a triangulation of $P$ at all. In fact, already in rank $3$ there are generic polygons $P$ such that $\conv(P)$ is trivial in the sense that $\conv(P)=P$. Hence, the subcomplex induced by $\conv(P)$ does not contain a triangulation of $P$. Such polygon configurations appear to be related to Morrison's Kekul\'e relations \cite{Mor}. 

We remark that the propositions in \S5 leading up to the proof of the main theorem provide a growth algorithm that, given a row-strict semistandard rectangular tableau as input, produces a non-elliptic web. It follows from the geometric results that this algorithm intertwines Sch\"utzenberger promotion on tableaux and rotation of webs. A growth algorithm yielding non-elliptic webs was also given by Kuperberg and Khovanov in \cite{KK}. That promotion on tableaux corresponds to rotation of webs under the Khovanov--Kuperberg algorithm was shown in \cite[Thm~2.5]{PPR}.    

Finally, we mention in passing that the geometry of point configurations in affine buildings is also related to higher laminations \cite{FG, Le1}. In \cite{Le2}, Le described higher laminations as positive configurations of points and interpreted the duality pairing of Fock--Goncharov in terms of length-minimal weighted networks. See \cite{Le1, Le2, LO}.

The paper is organized as follows. In \S\ref{sec:webs} we define non-elliptic webs and triangulated dual diskoids. In \S\ref{sec:building} we define the affine building, the convex sets, and prove a property of $\conv$ in this rank-2 setting. In \S\ref{sec:grassmannian} we define the polygon configuration space of points in the building and recall the result in \cite{FK,Akh} that describes the components in terms of cylindrical growth diagrams. In \S\ref{sec:proof} we give a proof of the main theorem using the geometry of the building, together with growth diagrams. It is the most involved section.

\subsection*{Acknowledgements}
I would like to thank Greg Kuperberg for helpful discussions and Leon Zhang for bringing convexity in affine buildings to my attention.

\section{$A_2$ Webs}\label{sec:webs}
The $A_2$ webs are defined as follows in \cite{Kup}.
\begin{Def}
A \emph{web} is a planar directed graph embedded in a disk (up to isotopy) with no multiple edges such that 
\begin{enumerate}
\item
each interior vertex is trivalent with all adjacent edges either pointing towards the vertex or away from the vertex
\item
each vertex on the boundary of the disk is univalent.
\end{enumerate}
A web is called \emph{non-elliptic} if every internal face contains at least $6$ sides.
\end{Def}

\begin{figure}
\centering
\begin{subfigure}[b]{.45\textwidth}
\center
\begin{tikzpicture}
\draw[decoration={markings, mark=at position .5 with {\arrow{>}}}, postaction={decorate}] 
(0,1) -- (0,0);
\end{tikzpicture}
\hspace{2mm}
\begin{tikzpicture}
\draw[decoration={markings, mark=at position .5 with {\arrow{>}}}, postaction={decorate}]
(0,0) -- (.5,.5);
\draw[decoration={markings, mark=at position .5 with {\arrow{>}}}, postaction={decorate}]
(1,0) -- (.5,.5);
\draw[decoration={markings, mark=at position .5 with {\arrow{>}}}, postaction={decorate}]
(.5,1) -- (.5,.5);
\end{tikzpicture}
\hspace{2mm}
\begin{tikzpicture}
\draw[decoration={markings, mark=at position .5 with {\arrow{>}}}, postaction={decorate}]
(.5,.5) -- (0,0);
\draw[decoration={markings, mark=at position .5 with {\arrow{>}}}, postaction={decorate}]
(.5,.5) -- (1,0);
\draw[decoration={markings, mark=at position .5 with {\arrow{>}}}, postaction={decorate}]
(.5,.5) -- (.5,1);
\end{tikzpicture}
\caption{Generators for the $\mathfrak{sl}_3$ spider. \label{fig:generators}}
\end{subfigure}
\begin{subfigure}[b]{.45\textwidth}
\center
\begin{align*}
\begin{tikzpicture}[baseline=-.5ex]
\draw[decoration={markings, mark=at position 0 with {\arrow{>}}}, postaction={decorate}] 
(0,0) circle (.4);
\end{tikzpicture} \;&=\;3\\
\begin{tikzpicture}[baseline=-.5ex]
\draw[decoration={markings, mark=at position .5 with {\arrow{>}}}, postaction={decorate}] 
(.5,0) -- (1,0);
\draw[decoration={markings, mark=at position .5 with {\arrow{>}}}, postaction={decorate}] 
(2,0) -- (2.5,0);
\draw[decoration={markings, mark=at position .5 with {\arrow{>}}}, postaction={decorate}] 
(2,0) arc (30:150:.58);
\draw[decoration={markings, mark=at position .5 with {\arrow{>}}}, postaction={decorate}] 
(2,0) arc (-30:-150:.58);
\end{tikzpicture} \;&=\;-2\; 
\begin{tikzpicture}[baseline=-.5ex]
\draw[decoration={markings, mark=at position .5 with {\arrow{>}}}, postaction={decorate}] 
(0,0) -- (1,0);
\end{tikzpicture}\\
\begin{tikzpicture}[baseline=-3.5ex]
\draw[decoration={markings, mark=at position .5 with {\arrow{>}}}, postaction={decorate}] 
(0,0) -- (1,0);
\draw[decoration={markings, mark=at position .5 with {\arrow{>}}}, postaction={decorate}] 
(1,-1) -- (1,0);
\draw[decoration={markings, mark=at position .5 with {\arrow{>}}}, postaction={decorate}] 
(0,0) -- (0,-1);
\draw[decoration={markings, mark=at position .5 with {\arrow{>}}}, postaction={decorate}] 
(1,-1) -- (0,-1);
\draw[decoration={markings, mark=at position .5 with {\arrow{>}}}, postaction={decorate}] 
(0,0) -- (-.5,.5) ;
\draw[decoration={markings, mark=at position .5 with {\arrow{>}}}, postaction={decorate}] 
(1.5,.5) -- (1,0);
\draw[decoration={markings, mark=at position .5 with {\arrow{>}}}, postaction={decorate}] 
(-.5,-1.5) -- (0,-1);
\draw[decoration={markings, mark=at position .5 with {\arrow{>}}}, postaction={decorate}] 
(1,-1) -- (1.5,-1.5);
\end{tikzpicture}\; &=\; 
\begin{tikzpicture}[baseline=-3.5ex]
\draw[decoration={markings, mark=at position .5 with {\arrow{>}}}, postaction={decorate}] 
(0,-1) arc (150:30:.58);
\draw[decoration={markings, mark=at position .5 with {\arrow{>}}}, postaction={decorate}] 
(1,0) arc (-30:-150:.58);
\end{tikzpicture}\;+\;
\begin{tikzpicture}[baseline=-3.5ex]
\draw[decoration={markings, mark=at position .5 with {\arrow{>}}}, postaction={decorate}] 
(1,0) arc (135:225:.707);
\draw[decoration={markings, mark=at position .5 with {\arrow{>}}}, postaction={decorate}] 
(0,-1) arc (-45:45:.707);
\end{tikzpicture}
\end{align*}
\caption{Relations for the $\mathfrak{sl}_3$ spider.\label{fig:relations}}
\end{subfigure}
\caption{Kuperberg's presentation of the $\mathfrak{sl}_3$ spider. \label{fig:spider}}
\end{figure}
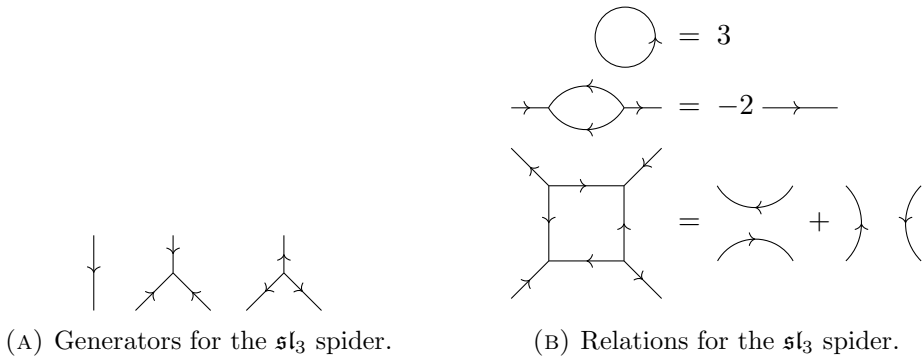

Fix a sequence of fundamental weights $\vec\lambda=(\lambda^1,\ldots,\lambda^n)$. One can view the spider category of $\mathfrak sl_3$ as being generated by two types of trivalent vertices and a directed strand, as shown in Figure \ref{fig:generators}. It is a fact that any invariant vector in the tensor product representation 
\begin{align*}
V_{\lambda^1}\otimes \cdots \otimes V_{\lambda^n}
\end{align*}
can be interpreted as a linear combination of webs with $i$th boundary vertex incident to an edge pointing away from the boundary if $\lambda^i=\omega_1$ and towards the boundary if $\lambda^i=\omega_2$. We say that the \emph{boundary} or \emph{type} of the web is $\vec\lambda$. Any web can be reduced to a linear combination of non-elliptic webs by applying the reduction relations shown in Figure \ref{fig:relations}. Hence, the non-elliptic webs with boundary $\vec\lambda$ span the invariant space $\Inv(\vec\lambda)=\Inv(V_{\lambda^1}\otimes\cdots\otimes V_{\lambda^n})$. Kuperberg \cite{Kup} showed that the number of non-elliptic webs is equal to $\dim\left(\Inv(\vec\lambda)\right)$, so they form a basis for the invariant space.

As mentioned in the introduction, there is a dual picture to webs, that of a diskoid, as defined in \cite{FKK}. It is the graphical dual of a web considered as a planar graph. More precisely, given a web, the dual graph has a vertex for each connected component of the disk. Connect vertices corresponding to neighboring components with a directed edge such that the direction is given by rotating counterclockwise the direction of the corresponding web edge. A web with $n$ external vertices partitions the boundary of the disk into $n$ arcs. Hence, the result of the dualizing procedure can be thought of as a triangulated $n$-gon with directed edges. Note that some of the $n$-gon vertices may be identified if their corresponding arcs are contained in the same connected component of the disk, resulting in a triangulated diskoid. See the example in Figure \ref{fig:dual examples}. A diskoid is CAT(0) if every interior vertex has degree at least $6$. Non-elliptic webs correspond precisely to CAT(0) diskoids. 

Consider a path from vertex $x$ to vertex $y$ in a diskoid that agrees with the orientation on $a$ of its edges and disagrees on $b$ of its edges. The dominant-weight-valued length of this path from $x$ to $y$ is defined to be $a\omega_1+b\omega_2$. Paths in the dual diskoid are dual to cut paths in the corresponding web as defined in \cite{Kup}.

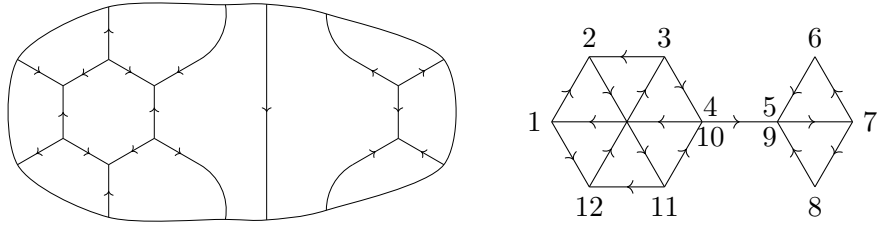
\begin{figure}
\scalemath{.7}{
\begin{tikzpicture}
\tikzmath{\a=.866;}
\draw[arrowto] (-\a,-.5) -- (-\a,.5);
\draw[arrowfrom] (-\a,.5) -- (0,1);
\draw[arrowto] (0,1) -- (\a,.5);
\draw[arrowfrom] (\a,.5) -- (\a,-.5);
\draw[arrowto] (\a,-.5) -- (0,-1);
\draw[arrowfrom] (0,-1) -- (-\a,-.5);

\draw[arrowfrom] (-2*\a,-1) -- (-\a,-.5);
\draw[arrowto] (-2*\a,1) -- (-\a,.5);
\draw[arrowfrom] (0,2) -- (0,1);
\draw[arrowfrom] (2*\a,-1) -- (\a,-.5);
\draw[arrowto] (2*\a,1) -- (\a,.5);
\draw[arrowto] (0,-2) -- (0,-1);

\draw (2*\a,1) arc (-60:10:1) coordinate (A);
\draw (2*\a,-1) arc (60:-10:1) coordinate(B);

\draw[arrowto] (3,2.05) -- (3,-2.05);

\tikzmath{\b=5.5;}
\draw[arrowto] (\b,.5) -- (\b,-.5);
\draw[arrowfrom] (\b-\a,1) -- (\b,.5);
\draw[arrowfrom] (\b+\a,1) -- (\b,.5);
\draw[arrowto] (\b-\a,-1) -- (\b,-.5);
\draw[arrowto] (\b+\a,-1) -- (\b,-.5);

\draw (\b-\a,1) arc (240:180:1) coordinate(C);
\draw (\b-\a,-1) arc (120:180:1) coordinate(F);

\draw plot [smooth cycle] coordinates {(-2*\a,1) (0,2) (A) (3,2.05) (C) (\b+\a,1) (\b+\a,-1) (F) (3,-2.05) (B) (0,-2) (-2*\a,-1)};
\end{tikzpicture}}
\hspace{5mm}
\begin{tikzpicture}
\tikzmath{\a=.866;}
\tikzmath{\a=.866;}
\draw[arrowto] (0,0) -- (-1,0);
\draw[arrowfrom] (0,0) -- (1,0);
\draw[arrowto] (0,0) -- (.5,\a);
\draw[arrowto] (0,0) -- (.5,-\a) coordinate(A);
\draw[arrowfrom] (0,0) -- (-.5,\a);
\draw[arrowfrom] (0,0) -- (-.5,-\a) coordinate(B);

\draw[arrowto] (-1,0) -- (-.5,\a);
\draw[arrowfrom] (-.5,\a) -- (.5,\a);
\draw[arrowto] (.5,\a) -- (1,0);
\draw[arrowfrom] (1,0) -- (.5,-\a);
\draw[arrowto] (.5,-\a) -- (-.5,-\a);
\draw[arrowfrom] (-.5,-\a) -- (-1,0);

\draw[arrowto] (1,0) -- (2,0);
\draw[arrowto] (2,0) -- (3,0);
\draw[arrowfrom] (2,0) -- (2.5,\a) coordinate(X);
\draw[arrowfrom] (2,0) -- (2.5,-\a) coordinate(Y);
\draw[arrowto] (3,0) -- (2.5,\a);
\draw[arrowto] (3,0) -- (2.5,-\a);

\node[left] at (-1, 0) {$1$};
\node[above] at (-.5, \a) {$2$};
\node[above] at (.5, \a) {$3$};
\node[xshift=1.1mm,yshift=2mm] at (1, 0) {$4$};
\node[xshift=-1mm,yshift=2mm] at (2, 0) {$5$};
\node[above] at (X) {$6$};
\node[right] at (3,0) {$7$};
\node[below] at (Y) {$8$};
\node[xshift=-1mm,yshift=-2mm] at (2,0) {$9$};
\node[xshift=1.1mm,yshift=-2mm] at (1,0) {$10$};
\node[below] at (A) {$11$};
\node[below] at (B) {$12$};
\end{tikzpicture}
\caption{A non-elliptic web of type $(1,2,1,1,2,2,1,1,2,2,1,2)$ and its triangulated dual diskoid. The web has one interior face. \label{fig:dual examples}}
\end{figure}

\section{The Affine Building}\label{sec:building}
\subsection{Definition of the Affine Building}
Let $\Delta_m$ denote the affine building of type $A_m$. It is an infinite simplicial complex of dimension $m$, which we now define (our reference is \cite[\S9.2]{Ron}). Consider the field of Laurent series $\mathcal K=\mathbb C((t))$. Let $U=\mathcal K^{m+1}$ be an $(m+1)$-dimensional vector space over $\mathcal K$. The field $\mathcal K$ has a discrete valuation with uniformizer $t$. Its discrete valuation ring is the ring of power series $\mathcal O=\mathbb C[[t]]$. An \emph{$\mathcal O$-lattice} $L$ is a finitely generated $\mathcal O$-submodule of $U$ such that the $\mathcal K$-span of $L$ generates all of $U$. Equivalently, $L$ is an $\mathcal O$-lattice if it is closed under multiplication by $t$, $L$ is contained in $t^{-k}\mathcal O^{m+1}$ for some large $k$, and $L$ contains $t^k\mathcal O^{m+1}$ for some large $k$. Here $\mathcal O^{m+1}=\vspan_{\mathcal O}(e_1,\ldots,e_{m+1})$ is the base lattice where $e_1,\ldots,e_{m+1}$ is a basis of $\mathcal K^{m+1}$. Two lattices $L,L'$ are \emph{homothetic}, or \emph{equivalent}, if $L'=t^kL$ for some integer $k$. Let $[L]$ denote the homothety class of the lattice $L$. 

The vertices of $\Delta_m$ are the homothety classes of lattices. The edges are unordered pairs of classes $\left(x,y\right)$ such that for any lattice $L$ in the class of $x$, there exists a lattice $L'$ in the class of $y$ such that $tL\subset L'\subset L$. Then $\Delta_m$ is the flag complex defined on this $1$-skeleton - that is, a set of vertices $x_1,\ldots,x_{k+1}$ forms a $k$-simplex if the $x_i$ are pairwise adjacent. For any simplex $(x_1,\ldots,x_{k+1})$ one can always find representatives $L_i$ of the $x_i$ such that 
\begin{align*}
tL_1\subset L_{k+1}\subset L_{k}\subset\cdots\subset L_2\subset L_1. 
\end{align*}
Note that $L_1/tL_1$ is a $\mathbb C$-vector space of dimension $m+1$ containing the subspaces $L_i/tL_1$. Hence the maximal simplices have $m+1$ vertices, so $\Delta_m$ is $m$-dimensional. 

There is a more axiomatic description of $\Delta_m$ as follows (see \cite[\S4.1]{AB}). The affine building $\Delta_m$ is a simplicial complex that is a union of subcomplexes called apartments satisfying the following.
\begin{itemize}
\item
Each apartment is the affine Coxeter complex of type $A_m$.
\item
For any two simplices $\sigma,\tau\in \Delta_m$ there exists an apartment containing both.
\item
If two apartments contain a common simplex, then there is a simplicial isomorphism between the two apartments that fixes all common points.
\end{itemize}

We will only be concerned with $\Delta_2$. In this case, each apartment is the Euclidean plane triangulated by equilateral triangles, the vertices of which form the $SL_3$ integral weight lattice. Recall that the weight lattice of $SL_3$ is $\mathbb Z^3/(1,1,1)$ (see Figure \ref{fig:weight lattice}). A choice of basis $v_1,v_2,v_3$ for $\mathcal K^3$ determines an apartment whose vertices are precisely the classes corresponding to lattices of the form $\vspan_{\mathcal O}\left(t^{-\mu_1}v_1,t^{-\mu_2}v_2,t^{-\mu_3}v_3\right)$ where $\mu_1,\mu_2,\mu_3$ are integers (and $(\mu_1,\mu_2,\mu_3)$ can be thought of as an $SL_3$-weight). 

There is a dominant-weight-valued metric on the vertices of the building defined as follows. Any two vertices $x,y$ in $\Delta_2$ are contained in a common apartment by the second bullet point above. Hence, we can choose a basis $v_1,v_2,v_3$ for $\mathcal K^3$ such that $x$ is the class of $\vspan_{\mathcal O}\left(v_1,v_2,v_3\right)$ and $y$ is the class of $\vspan_{\mathcal O}\left(t^{-\mu_1}v_1,t^{-\mu_2}v_2,t^{-\mu_3}v_3\right)$ for integers $\mu_1\geq \mu_2 \geq \mu_3$. In other words, $\mu=(\mu_1,\mu_2,\mu_3)$ is a dominant weight. Define the \emph{distance} between $x$ and $y$ to be $d(x,y)=\mu$. It is well defined, in the sense that if $w_1,w_2,w_3$ is any other basis such that $x=\vspan_{\mathcal O}\left(w_1,w_2,w_3\right)$ and $y=\vspan_{\mathcal O}\left(t^{-\lambda_1}w_1,t^{-\lambda_2}w_2,t^{-\lambda_3}w_3\right)$ for which $\lambda_1\geq \lambda_2\geq \lambda_3$, then $\lambda=\mu$. There is the usual partial order on dominant weights, so that $\lambda\geq \mu$ if $\lambda-\mu$ is a non-negative integer combination of simple roots. 

A \emph{combinatorial geodesic} between $x$ and $y$ will mean any minimal path between~$x$ and $y$ along the edges of $\Delta_2$. It is a fact that any apartment containing both $x$ and $y$ also contains every combinatorial geodesic between $x$ and $y$. In any such apartment, the set of combinatorial geodesics span a parallelogram (see Figure \ref{fig:weight lattice}). 

Although we will not need it, note that there is also a locally Euclidean distance on $\Delta_2$ defined by taking the usual Euclidean distance in any apartment containing two points $p,q$ of the building. This makes the building into a CAT(0) space \cite{BT}. The local Euclidean geodesic is the usual Euclidean line segment passing through the parallelogram of combinatorial geodesics. See \cite[\S 11, 12]{AB} for more details on this point of view.

Recall that the fundamental weights are $\omega_1=(1,0,0)$ and $\omega_2=(1,1,0)$. A directed edge from $x$ to $y$ will be called an \emph{$\omega_1$-step} (resp. \emph{$\omega_2$-step}) if $d(x,y)=\omega_1$ (resp. $\omega_2$). Note that $d(x,y)=\omega_1$ if and only if $d(y,x)=\omega_2$. More generally, $d(y,x)=d(x,y)^*$ where the dual of a dominant weight $\lambda$ is $\lambda^*=(-\lambda_3,-\lambda_2,-\lambda_1)$. For a dominant weight $\lambda=a\omega_1+b\omega_2$, let $\lvert \lambda\rvert=a+b$ denote the number of steps in any combinatorial geodesic. This is the minimum number of steps in a path from $[L]$ to $[L']$ if $d([L],[L'])=\lambda$. 

We will often use the following fact. Given three distinct vertices $x,y,z$ in the building such that $d(x,y)=\omega_1$ and $d=(y,z)=\omega_2$ there exists a unique vertex adjacent to all three $x,y,z$ (likewise if $d(x,y)=\omega_2$ and $d=(y,z)=\omega_1$).

\begin{figure}
\begin{tikzpicture}
\begin{rootSystem}{A}
\weightLattice{5}
\tikzmath{\x=3; \y=1;}
\tikzmath{\a=-2; \b=-2;}
\tikzmath{\w=\x-\a-1;\h=\y-\b-1;}

\draw[thick] \weight{\a}{\b} -- \weight{\a}{\y};
\draw[thick] \weight{\a}{\y} -- \weight{\x}{\y};

\draw[thick] \weight{\a}{\b} -- \weight{\x}{\b};
\draw[thick] \weight{\x}{\b} -- \weight{\x}{\y};

\tikzmath{\H=\h+1;\W=\w+1;}
\foreach \i in {1,...,\H}
{\draw[thick, arrowto, red] \weight{\a}{\b+\i-1} -- \weight{\a}{\b+\i};
\draw[thick, arrowto, blue] \weight{\x}{\b+\i-1} -- \weight{\x}{\b+\i};}

\node[above left, red] at (hex cs:x=-2,y=1){\small\(\minconv([L],[L'])\)};
\node[below right, blue] at (hex cs:x=3,y=-2){\small\(\maxconv([L],[L'])\)};

\foreach \i in {1,...,\W}
{\draw[thick, arrowto, blue] \weight{\a+\i-1}{\b} -- \weight{\a+\i}{\b};
\draw[thick, arrowto, red] \weight{\a+\i-1}{\y} -- \weight{\a+\i}{\y};}

\foreach \i in {1,...,\w}
\draw[thick] \weight{\a+\i}{\b} -- \weight{\a+\i}{\y};
\foreach \i in {1,...,\h}
\draw[thick] \weight{\a}{\b+\i} -- \weight{\x}{\b+\i};

\node[below left=-3pt] at (hex cs:x=-2,y=-2){\small\([L]\)};
\node[above right=-2pt] at (hex cs:x=3,y=1){\small\([L']\)};

\end{rootSystem}
\end{tikzpicture}
\begin{tikzpicture}
\begin{rootSystem}{A}
\weightLattice{2}
\draw[arrowto, thick] \weight{0}{0} -- \weight{1}{0};
\draw[arrowto, thick] \weight{0}{0} -- \weight{-1}{1};
\draw[arrowto, thick] \weight{0}{0} -- \weight{0}{-1};

\node[below] at (0,-1) {$\omega_1$-steps};

\end{rootSystem}
\end{tikzpicture}
\begin{tikzpicture}
\begin{rootSystem}{A}
\weightLattice{2}
\draw[arrowto, thick] \weight{0}{0} -- \weight{0}{1};
\draw[arrowto, thick] \weight{0}{0} -- \weight{-1}{0};
\draw[arrowto, thick] \weight{0}{0} -- \weight{1}{-1};

\node[below] at (0,-1) {$\omega_2$-steps};

\end{rootSystem}
\end{tikzpicture}
\caption{A portion of the $A_2$ weight lattice - a single apartment in $\Delta_2$. The combinatorial geodesics forming a parallelogram between two vertices are shown in bold. Arrows are depicted along the two geodesics consisting of vertices in $\minconv$ and $\maxconv$. \label{fig:weight lattice}}
\end{figure}
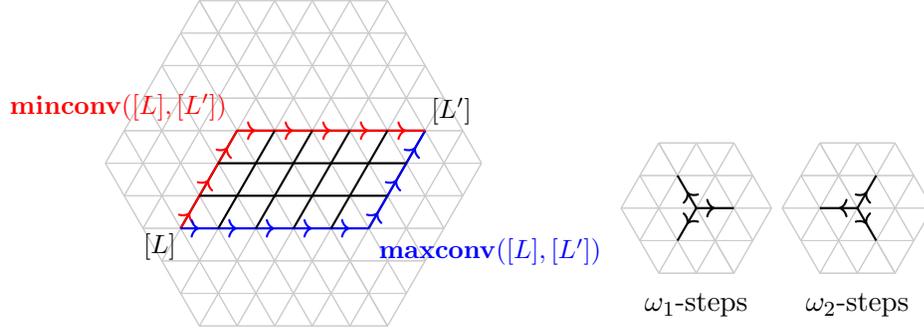

\subsection{Convex Sets}
We now define convex sets in the affine building. See \cite{JSY, Zha} for an exposition of the following definitions, as well as their connection to tropical convexity. The tropical geometry point of view is not used in the present paper.

There are two dual notions of convexity. For two lattices $L,L'$ their intersection $L\cap L'$ is again a lattice. Likewise, their sum $L+L'$ is a lattice. A set of lattice classes is \emph{min-convex} (resp. \emph{max-convex}) if it closed under taking intersections (resp. sums) of a finite subset of lattice representatives. That is, a set $X$ of lattice classes is min-convex if for any finite set of lattice classes $[L_1],\ldots,[L_k]$ in $X$, the class $[L_1\cap\cdots\cap L_k]$ is in $X$. The min and max terminology comes from the interpretation of the affine building in terms of additive norms on $\mathcal K^3$, as mentioned in \cite[\S2]{JSY}.

\begin{Def}
\begin{itemize}
\item
Let $\minconv([L_1],\ldots,[L_k])$ denote the min-convex hull of $[L_1],\ldots,[L_k]$. It is the smallest min-convex set containing the lattice classes $[L_1],\ldots,[L_k]$.
\item
Let $\maxconv([L_1],\ldots,[L_k])$ denote max-convex hull of $[L_1],\ldots,[L_k]$. 
 \item
Let $\conv([L_1],\ldots,[L_k])$ denote the intersection \\$\minconv([L_1],\ldots,[L_k])\cap \maxconv([L_1],\ldots,[L_k])$ of the two convex sets.
\end{itemize}
\end{Def}

The min-convex hull of a finite number of lattice classes is finite, a result attributed to Faltings \cite{Fal} in \cite{KT}. The following lemma and proof of the following proposition appears in \cite[2.11, 2.12]{Zha}. 

\begin{Lem}\label{lem:pairs}
Let $[L_1],\ldots,[L_k]$ be a collection of lattice classes. Then
\begin{align*}
\minconv([L_1],\ldots,[L_k])=\bigcup_{[L]\in \minconv([L_2],\ldots,[L_k])} \minconv([L_1],[L]).
\end{align*}
\end{Lem}

\begin{Prop}
Let $[L_1],\ldots, [L_k]$ be lattice classes representing vertices in $\Delta_m$. The sets $\minconv([L_1],\ldots,[L_k])$ and $\maxconv([L_1],\ldots,[L_k])$ are finite. 
\end{Prop}
\begin{proof}
For two lattices $L_1,L_2$ the set $\minconv([L_1],[L_2])$ consists of classes with representatives of the form $L_1\cap t^aL_2$. For large enough $a$ we have $L_1\supseteq t^aL_2$ and $L_1\subseteq t^{-a}L_2$, so $\minconv([L_1],[L_2])$ is finite. The proposition follows by induction together with the previous lemma.
\end{proof}

Restrict to $\Delta_2$. Then $\minconv([L],[L'])$ is the set of vertices on the combinatorial geodesic from $[L]$ to $[L']$ consisting of $\omega_2$ steps followed by $\omega_1$ steps. This path forms the upper boundary of the parallelogram in Figure \ref{fig:weight lattice}. It is the combinatorial geodesic from $[L']$ to $[L]$ that also first takes $\omega_2$ steps followed by $\omega_1$ steps, since $\omega_1^*=\omega_2$. 

Similarly, $\maxconv([L],[L'])$ is the set of vertices on the combinatorial geodesic from $[L]$ to $[L']$ (or from $[L']$ to $[L]$) consisting of $\omega_1$ steps followed by $\omega_2$ steps. It forms the lower boundary of the parallelogram in Figure \ref{fig:weight lattice}. Note that $\minconv([L],[L'])$ and $\maxconv([L],[L'])$ are equal if $d([L],[L'])$ is a multiple of $\omega_1$ or $\omega_2$. If this is not the case, then the term \emph{elbow vertex} will mean the terminal vertex of the last $\omega_2$ step (resp. $\omega_1$) along the $\minconv([L],[L'])$ geodesic (resp. $\maxconv([L],[L'])$).

\subsection{A Property of $\conv$ for Paths}

Since we will be interested in paths in the affine building, and more specifically polygons, we will need the following lemma and proposition. 
\begin{Lem}\label{lem:pairwise}
If $P=\{[L_1],\ldots,[L_n]\}$ is a path in $\Delta_2$, then
\begin{align*}
\minconv(P)&=\cup_{i\not=j}\minconv([L_i],[L_j]),\\
\maxconv(P)&=\cup_{i\not=j}\maxconv([L_i],[L_j]).
\end{align*}
\end{Lem}
\begin{proof}
Induct on the length of the path $n$. The base case $n=2$ is trivial, so suppose that $n>2$ and that the result holds for paths of length less than $n$. By induction $\minconv([L_1],\ldots,[L_{n-1}])=\cup_{1\leq i<j\leq n-1}\minconv([L_i],[L_j])$. Then by Lemma \ref{lem:pairs}
\begin{align*}
\minconv([L_1],\ldots,[L_n])&=\bigcup_{[K]\in \minconv([L_1],\ldots,[L_{n-1}])}\minconv([K],[L_n])\\
&=\bigcup_{[K]\in \bigcup\minconv([L_i],[L_j])}\minconv([K],[L_n]).\\
&=\bigcup_{1\leq i<j\leq n-1}\minconv([L_i],[L_j],[L_n])
\end{align*}
If $i\not=1$, then $\minconv([L_i],[L_j],[L_n])$ is contained in 
\begin{align*}
\minconv([L_2],\ldots,[L_n])=\bigcup_{2\leq l<k\leq n}\minconv([L_l],[L_k])
\end{align*}
where the equality is by induction. So it remains to show that for any fixed $j$ that $\minconv([L_1],[L_j],[L_n])$ is contained in $\cup_{l\not=k}\minconv([L_l],[L_k])$.

Let $[L]$ be an element of $\minconv([L_1],[L_j],[L_n])$, so that there are integers $a$ and $b$ such that $[L]=[L_1\cap t^aL_j\cap t^b L_n]$. Consider the set $\minconv([L_1],[L_n])$ as a path in the building. Then the vertex $[L]$ lies on a path $\minconv([K],[L_j])$ for some $[K]\in \minconv([L_1],[L_n])$.

Consider the two $1$-simplices $([L_1],[L_2])$ and $([L_n],[L_{n-1}])$, and let $A$ be an apartment containing both. Then $A$ contains all of the combinatorial geodesics between $[L_1]$ and $[L_n]$, and hence also contains $\minconv([L_1],[L_n])$. Let $v_1,v_2,v_3$ be a basis defining $A$ such that $L_1=\vspan_{\mathcal O}(v_1,v_2,v_3)$ and $L_n=\vspan_{\mathcal O}(t^{-x}v_1,t^{-y}v_2,v_3)$ for integers $x\geq y\geq 0$. Then the path $\minconv([L_1],[L_n])$ from $[L_1]$ to $[L_n]$ is a sequence of $y$ many $\omega_2$ steps followed by $x-y$ many $\omega_1$ steps. 

There are six possible configurations for $[L_2]$ and six for $[L_{n-1}]$ in apartment $A$ relative to the path $\minconv([L_1],[L_n])$, as depicted in Figure \ref{fig:apartment}. Note that if $[K]$ is contained in $\minconv([L_2],[L_n])$ (resp. $\minconv([L_1],[L_{n-1}])$), then $[L]$ is in $\minconv([L_2],[L_j],[L_n])$ (resp. $\minconv([L_1],[L_j],[L_{n-1}])$), and the result follows by induction.

If $y=0$, then $[K]$ is contained in $\minconv([L_2],[L_n])$ for any choice of $[L_2]$. Similarly, if $x-y=0$, then $[K]$ is contained in $\minconv([L_1],[L_{n-1}])$. Hence, we may assume that both $y\not=0$ and $x-y\not=0$. 

For $[K]$ to not be contained in $\minconv([L_2],[L_n])$ nor $\minconv([L_1],[L_{n-1}])$ we must have that $[K]$ is the elbow vertex $\vspan_{\mathcal O}\left(t^{-y}v_1,t^{-y}v_2,v_3\right)$ and $[L_2]$ (resp. $[L_{n-1}]$) is one of $\vspan_{\mathcal O}\left(t^{-1}v_1,v_2,v_3\right)$ or $\vspan_{\mathcal O}\left(t^{-1}v_1,v_2,t^{-1}v_3\right)$ (resp. $\vspan_{\mathcal O}\left(t^{-x+1}v_1,t^{-y+1}v_2,v_3\right)$ or $\vspan_{\mathcal O}\left(t^{-x}v_1t^{-y+1},v_2,v_3\right)$). Let $a$ (resp. $b$) denote the vertex adjacent to $[K]$ that is on the geodesic from $[K]$ to $[L_1]$ (resp. $[L_n]$). Let $c$ denote the vertex that is simultaneously adjacent to all three $[K], a$, and $b$. See Figure \ref{fig:apartment}. 

Suppose that $d([L_j],[K])=s\omega_1+t\omega_2$ for nonnegative integers $s$ and $t$, so that $\minconv([L_j],[K])$ is a path from $[L_j]$ to $[K]$ consisting of $t$ many $\omega_2$ steps followed by $s$ many $\omega_1$ steps. Suppose that $s=0$. Let $e=(d,[K])$ be the last edge of the path $\minconv([L_j],[K])$, which is an $\omega_2$ step into $[K]$. If $d=a$, then all of $\minconv([L_j],[K])$ is contained in $\minconv([L_j],[L_n])$, and in particular $[L]\in \minconv([L_j],[L_n])$ so the result follows. If $d\not=a$, then the edges $(d,[K])$ and $(a,[K])$ can be placed in a common apartment from which it can be seen that $\minconv([L_j],[L_1])$ consists of the $\omega_2$ steps from $[L_j]$ to $[K]$ followed by the $\omega_1$ steps from $[K]$ to $[L_1]$. Hence, $\minconv([L_j],[K])$, and in particular $[L]$, is contained in $\minconv([L_j],[L_1])$.

Now suppose that $s\not=0$, so that the final edge $e=(d,[K])$ of $\minconv([L_j],[K])$ is an $\omega_1$ step into the vertex $[K]$. Place $e$ and $(a,[K])$ into a common apartment. If $d$ and $a$ are not adjacent vertices in the building, then the geodesic from $d$ to $[L_1]$ consists of $\omega_1$ steps, so $\minconv([L_j],[K])$ is contained in $\minconv([L_j],[L_1])$ and the result follows. Now suppose that $d$ and $a$ are adjacent in the building. Suppose that $d=c$. Note that $c$ is contained in $\minconv([L_2],[L_n])$. Hence, $[L]$ is contained in $\minconv([L_2],[L_j],[L_n])$, so the result follows by induction. If $d\not=c$, then the geodesic from $d$ to $[L_n]$ contains the geodesic from $[K]$ to $[L_n]$, so $\minconv([L_j],[K])$ is contained in $\minconv([L_j],[L_n])$.
\end{proof}
\begin{figure}
\center
\begin{tikzpicture}
\tikzmath{\x=3; \y=1;}
\tikzmath{\a=-2; \b=-2;}

\begin{rootSystem}{A}
\draw[thick] \weight{\a}{\b} -- \weight{\a}{\y};
\draw[thick] \weight{\a}{\y} -- \weight{\x}{\y};
\wt {\a}{\b}
\wt [red] {\a+1}{\b}
\wt {\a}{\b+1}
\wt {\a-1}{\b}
\wt {\a}{\b-1}
\wt [red] {\a+1}{\b-1}
\wt {\a-1}{\b+1}

\node[below right, black] at (hex cs:x=-1,y=-2){\small\([L_2]\)};
\node[below right, black] at (hex cs:x=-1,y=-3){\small\([L_2]\)};

\node[below, black] at (hex cs:x=3,y=0){\small\([L_{n-1}]\)};
\node[below right, black] at (hex cs:x=4,y=0){\small\([L_{n-1}]\)};

\node[left, black] at (hex cs:x=-2,y=0){\small\(a\)};
\node[above left, black] at (hex cs:x=-2,y=1){\small\([K]\)};
\node[above right, black] at (hex cs:x=-1,y=1){\small\(b\)};
\node[right, black] at (hex cs:x=-1,y=0){\small\(c\)};

\wt {\x}{\y}
\wt {\x+1}{\y}
\wt {\x}{\y+1}
\wt {\x-1}{\y}
\wt [red] {\x}{\y-1}
\wt [red] {\x+1}{\y-1}
\wt {\x-1}{\y+1}
\end{rootSystem}
\end{tikzpicture}
\caption{An apartment containing the path $\minconv\left([L_1],[L_n]\right)$ and the $1$-simplices $\left([L_1],[L_2]\right)$ and $\left([L_{n-1}],[L_n]\right)$. \label{fig:apartment}}
\end{figure}
Suppose that $d(v,v')=s\omega_1+t\omega_2$ for two vertices $v, v'$ in the building and nonnegative integers $s$ and $t$. Then $s$ and $t$ are both nonzero if and only if $\conv(v,v')=\left\{v,v'\right\}$. In this case, we say that $\conv(v,v')$ is trivial. If $s=0$ or $t=0$, then $\conv(v,v')=\minconv(v,v')=\maxconv(v,v')$ is a \emph{straight-path geodesic}, in which case the parallelogram of geodesics is simply a line segment. For a path $P$ in the building the following proposition says that every vertex in $\conv(P)$ lies on a straight-path geodesic between two vertices of $P$.

\begin{Prop}\label{prop:conv}
If $P=\{[L_1],\ldots,[L_n]\}$ is a path in $\Delta_2$, then
\begin{align*}
\conv(P)=\cup_{i\not=j}\conv([L_i],[L_j]).
\end{align*}
\end{Prop}
\begin{proof}
Suppose that $[L]$ is an element of $\conv(P)$, so that by Lemma \ref{lem:pairwise} there are indices $i<j$ and $k<l$ and integers $x,y$ such that $[L]=[L_i\cap t^xL_j]=[L_k+t^yL_l]$. As in the previous lemma, consider an apartment containing the $1$-simplices $([L_i],[L_{i+1}])$ and $([L_j],[L_{j-1}])$ and the path $\minconv([L_i],[L_j])$. If $[L]$ is contained in some $\minconv([L_s],[L_t])$ that is a straight-path geodesic, then the result follows, so suppose that $\minconv([L_i],[L_j])$ has an elbow vertex $[K]$. If $[L]$ is not equal to $[K]$, then $[L]$ is contained in one of $\minconv([L_i],[L_{j-1}])$ or $\minconv([L_{i+1}],[L_j])$. By walking the $i$ and $j$ vertices towards each other, since $j-i$ bounds the number of steps in $\minconv([L_i],[L_j])$, there are two indices $i',j'$ such that either $\minconv([L_{i'}],[L_{j'}])$ is a straight-path geodesic, or that $[L]$ is the elbow vertex. Rename $i$ and $j$ to be these indices. Likewise, assume that $[L]$ is the elbow vertex in $\maxconv([L_k],[L_l])$. 

Let $a$ (resp. b) be the vertex adjacent to $[L]$ on the geodesic from $[L_i]$ (resp. $[L_j]$) to $[L]$. Similarly, let $c$ (resp. $d$) be the vertex adjacent to $[L]$ on the geodesic from $[L_k]$ (resp. $[L_l]$) to $[L]$. If $c$ is not adjacent to $a$, then $\minconv([L_k],[L_i])$ is a straight-path geodesic passing through $[L]$ consisting of $\omega_1$ steps from $[L_k]$ to $[L_i]$. Likewise, if $d$ is not adjacent to $a$, then $\minconv([L_l],[L_i])$ is a straight-path geodesic passing through $[L]$. If both $c$ and $d$ are adjacent to $a$, then they can't both be adjacent to $b$. Say $c$ is not adjacent to $b$, in which case $\minconv([L_k],[L_j])$ is a straight-path geodesic passing through $[L]$, consisting of $\omega_1$ steps from $[L_k]$ to $[L_j]$.
\end{proof}

\section{Growth Diagrams and Polygon Configurations}\label{sec:grassmannian}
In this section we define generic polygons in the affine building. Although we're interested in $m=2$, we'll state everything for general $m$. 
\subsection{The Polygon Space}
As mentioned in the introduction, the invariant space $\Inv(V_{\lambda^1}\otimes\cdots\otimes V_{\lambda^n})$ has an interpretation in terms of the geometric Satake correspondence. The geometric Satake correspondence describes the representation theory of $G$ in terms of perverse sheaves on the affine Grassmannian $Gr_{G^L}$ of the Langlands dual group $G^L$. The affine Grassmannian is an inductive limit of varieties and can be identified with the quotient $Gr_{G^L}=G^L(\mathcal K)/G^L(\mathcal O)$.

In the present case, $G=SL_{m+1}$, $G^L=PGL_{m+1}$, and the affine Grassmannian 
\begin{align*}
Gr_{PGL_{m+1}}=PGL_{m+1}(\mathcal K)/PGL_{m+1}(\mathcal O)
\end{align*}
can be identified with the set of lattice classes as follows. Let $e_1,\ldots,e_{m+1}$ be a basis of $\mathcal K^{m+1}$. Then $PGL_{m+1}(\mathcal K)$ acts transitively on lattice classes, and $PGL_{m+1}(\mathcal O)$ stabilizes the class of the base lattice $L_0=span_{\mathcal O}(e_1,\ldots,e_{m+1})$. Hence, the affine Grassmannian can be thought of as the vertices of the building~$\Delta_m$. 

For a sequence of fundamental weights $\vec\lambda=(\lambda^1,\ldots,\lambda^n)$ define the polygon configuration space as follows. Here $L_0$ denotes the base lattice.
\begin{align*}
\Poly(\vec\lambda)=\left\{\left([L_1]=[L_{n+1}],[L_2],\ldots,[L_n]\right)\in Gr_{PGL_{m+1}}^{n}:\,
\begin{aligned}
&[L_1]=[L_0],\\
&d([L_i],[L_{i+1}])=\lambda^i
\end{aligned}\right\}
\end{align*}
This forms an equidimensional, reducible variety, as shown by Haines in \cite{Hai1, Hai2}. It is also known as the Satake fiber. As a corollary of the geometric Satake correspondence, the number of irreducible components of $\Poly(\vec\lambda)$ is equal to the dimension of $\Inv_{SL_{m+1}}(\vec\lambda)$. We state this result here. We will recall a combinatorial description of the components in terms of distances in the next section. See \cite[\S 4.3]{HS} for a proof of the following theorem.

\begin{Thm}[\cite{Lus,Gin,BD,MV}]\label{Thm:Satake}
Under the geometric Satake correspondence, there is an isomorphism 
\begin{align*}
\left(V_{\lambda^1}\otimes \cdots \otimes V_{\lambda^n}\right)^{G^L}\cong H_{\text{top}}(\text{Poly}(\vec\lambda))
\end{align*}
where $H_{\text{top}}(\text{Poly}(\vec\lambda))$ is the top Borel--Moore homology of Poly$(\vec\lambda)$. Hence, the set of top components of Poly$(\vec\lambda)$ give a basis for $\left(V_{\lambda^1}\otimes \cdots V_{\lambda^n}\right)^G$.
\end{Thm}

\subsection{The Growth Diagram Associated to a Generic Polygon}
The main combinatorial tool that we will use in the proofs is that of growth diagrams. In what follows we will identify partitions with their Young diagrams. To define cylindrical growth diagrams consider the following staircase-shaped set of lattice points indexed by pairs of integers $(i,j)$. Here we consider $(i,j)$ to be the lattice point in the $i$th row and $j$th column with indices increasing down and to the right as in matrix notation. 
\begin{align*}
St_n=\{(i,j)\in \mathbb Z^2\mid 1\leq i\leq n+1, i\leq j\leq i+n\}
\end{align*}
\begin{Def}
Fix a rectangular partition $\pi=(k,\ldots,k)$ with $m+1$ rows. A cylindrical growth diagram for $SL_{m+1}$ of shape $\pi$ is an assignment of partitions $\gamma_{i,j}$, one for each $(i,j)$ in $St_n$, such that
\begin{itemize}
\item
the skew shapes of neighboring partitions $\gamma_{i,j}/\gamma_{i+1,j}$ and $\gamma_{i,j}/\gamma_{i,j-1}$ are vertical strips,
\item
$\gamma_{i,i}$ is the empty partition for all $i$, 
\item
$\gamma_{i,n+i}=\pi$ for all $i$, 
\item
and for each unit square the following local condition is satisfied.
\begin{align}\label{local condition}
\gamma_{i+1,j+1}=\sort(\gamma_{i+1,j}+\gamma_{i,j+1}-\gamma_{i,j})
\end{align}
\end{itemize}
The \emph{type} of a cylindrical growth diagram is $(\gamma_{1,2},\gamma_{2,3},\ldots,\gamma_{n,n+1})$. 
\end{Def}
These diagrams were originally defined in \cite{Spe} and \cite{Whi}, and studied in the $GL_{m+1}$ case in \cite{Akh}. See Example \ref{ex:octagon} below. Although not immediately obvious from the definition, cylindrical growth diagrams are periodic in the sense that $\gamma_{1,j}=\gamma_{1+n,j+n}$. As such, it is often convenient to extend the diagram to an infinite staircase of width $n$ where the index $i$ can be arbitrary, for fixed $i$ the index $j$ satisfies $i\leq j\leq i+n$, and the partition labels are extended periodically $\gamma_{i,j}=\gamma_{i+n,j+n}$. If the partitions $\gamma_{i,i},\ldots,\gamma_{i,i+n}$ along a row are interpreted as row-strict semistandard tableaux, then the local condition is the Bender--Knuth move, and successive rows are given by the promotion operator on row-strict semistandard tableaux. 

Note that the local condition can be used to fill in a partially labelled growth diagram from the northwest to the southeast. In particular, any row determines the entire diagram (using the periodicity property alluded to). The southeastern vertex label $\gamma_{i+1,j+1}$ of a unit square is determined by the other three labels. It is convenient to write the local condition as $\gamma_{i+1,j+1}=\sort(\gamma_{i,j+1}-(\gamma_{i,j}-\gamma_{i+1,j}))$, which can be interpreted as follows. The partitions $\gamma_{i,j}$, $\gamma_{i+1,j}$ differ by $1$ in some subset of positions $X$. Subtract~$1$ from $\gamma_{i,j+1}$ in the same positions $X$. This may not be a dominant weight, so sort it to make it dominant.

In what follows we will identify dominant $SL_{m+1}$-weights with equivalence classes of partitions containing at most $m+1$ parts. Since $SL_{m+1}$-weights are defined up to the all-ones vector $(1,\ldots,1)$, partitions differing by a column of height $m+1$ correspond to the same $SL_{m+1}$-weight. 

The following result relates the combinatorics of growth diagrams to the components of $\Poly(\vec\lambda)$. It is implicit in the work of Fontaine and Kamnitzer \cite{FK} and explained explicitly in the case of $GL_{m+1}$ in \cite{Akh}.

\begin{Thm}[\cite{FK, Akh}]\label{Thm:components}
Each component of $\Poly(\vec\lambda)$ contains an open dense set of points $P=\left([L_1],\ldots,[L_n]\right)$ such that the distances $d([L_i],[L_j])$ do not depend on $P$. There is a unique choice of representative partitions $\gamma_{i,j}$ for the $d([L_i],[L_j])$ such that they form a cylindrical growth diagram of type $\vec\lambda$ (here the $i,j$ indices in $d([L_i],[L_j])$ should be taken modulo n). Furthermore, this is a bijection between the components of $\Poly(\vec\lambda)$ and cylindrical growth diagrams of type $\vec\lambda$.
\end{Thm}

Let us now restrict to $m=2$. For a dominant weight $\lambda=a\omega_1+b\omega_2$ let $\lvert \lambda\rvert=a+b$, and for a partition $\gamma$ representing $\lambda$ define $\lvert\gamma\rvert=\lvert\lambda\rvert$. If $\lambda$ is a weight and $\gamma=(a,b,c)$ is a partition, we will often abuse notation and write $\lambda+\gamma$ to mean the corresponding weight $\lambda+(a,b,c)$. Likewise although $\omega_1$ and $\omega_2$ are weights, we will often also think of them as the partitions $(1,0,0)$ and $(1,1,0)$. Partitions should be thought of as specific representatives for weights, so that one gets a growth diagram. We will also use $\omega_3$ to denote the partition $(1,1,1)$, which is a representative for the zero weight.

\begin{Def}
We will say that a polygon $P=\left([L_1],\ldots,[L_n]\right)$ in $\Delta_2$ is \emph{generic} if it lies in one of the open dense sets of Theorem \ref{Thm:components}. Equivalently, the distances $d([L_i],[L_j])$ form a cylindrical growth diagram (for the unique choice of partition representatives $\gamma_{i,j}$).
\end{Def}

Note that since $d([L_j],[L_i])=d([L_i],[L_j])^*$, any growth diagram is determined by the triangle of partitions $\gamma_{i,j}$ for indices $1\leq i\leq n+1$ and $i\leq j\leq n+1$. In terms of partitions, this means that $\gamma_{j,i+n}=\gamma_{i,j}^c$, where $\gamma_{i,j}^c$ is gotten by rotating $\gamma_{i,j}$ by 180 degrees, placing it in the bottom right corner of the full rectangular Young diagram, and taking its complement in the rectangular Young diagram.

For two adjacent entries of a growth diagram, define $\dif(\gamma_{i,j},\gamma_{i+1,j})$ to be the indices of the rows in which the partitions differ. Hence it is a set of size $1$ or~$2$. This set can be seen as labelling the edge between the vertices $(i,j)$ and $(i+1,j)$ of the growth diagram, and likewise for the edge between vertices $(i,j)$ and $(i,j+1)$. We will need the following observation about growth diagrams. Define the orders $\{1\}<\{2\}<\{3\}$ and $\{1,2\}<\{1,3\}<\{2,3\}$ on size-$1$ and size-$2$ subsets of $\{1,2,3\}$ respectively.

\begin{Prop}\label{prop:difs}
\begin{enumerate}
Let $\{\gamma_{i,j}\}_{(i,j)\in St_n}$ be a cylindrical growth diagram.
\item
Suppose that $\gamma_{i,i+1}=\omega_1$. Then
\begin{itemize}
\item
$\dif(\gamma_{i,i+1},\gamma_{i+1,i+1})=\{1\}$, 
\item
$\dif(\gamma_{i,i+n},\gamma_{i+1,i+n})=\{3\}$, and 
\item
$\dif(\gamma_{i,j},\gamma_{i+1,j})$ is a set of size $1$ for all $j$ and increases monotonically from $1$ to $3$ as $j$ increases from $i+1$ to $i+n$.
\end{itemize}
\item
Suppose that $\gamma_{i,i+1}=\omega_2$. Then 
\begin{itemize}
\item
$\dif(\gamma_{i,i+1},\gamma_{i+1,i+1})=\{1,2\}$, 
\item
$\dif(\gamma_{i,i+n},\gamma_{i+1,i+n})=\{2,3\}$, and 
\item
$\dif(\gamma_{i,j},\gamma_{i+1,j})$ is a set of size $2$ for all $j$ and increases monotonically from $\{1,2\}$ to $\{2,3\}$ as $j$ increases from $i+1$ to $i+n$.
\end{itemize}
\end{enumerate}
\end{Prop}
\begin{proof}
We will prove part (a) with part (b) being similar. The first bullet point follows since $\gamma_{i,i+1}=\omega_1$ and $\gamma_{i+1,i+1}=\emptyset$. For simplicity assume that $i=1$. The local condition can be written as $\gamma_{2,j+1}=\sort\left(\gamma_{1,j+1}-(\gamma_{1,j}-\gamma_{2,j})\right)$. This can be interpreted as subtracting one from $\gamma_{1,j+1}$ in position $\dif(\gamma_{1,j},\gamma_{2,j})$ and sorting to get $\gamma_{2,j+1}$. Then $\dif(\gamma_{1,j+1},\gamma_{2,j+1})$ must also have size $1$. The resulting index in $\dif(\gamma_{1,j+1},\gamma_{2,j+1})$ must be equal to or larger than $\dif(\gamma_{1,j},\gamma_{2,j})$. This establishes the third bullet point. The second bullet point follows from the third and because $\gamma_{1,n+1}$ is a rectangle.
\end{proof}

\begin{Ex}\label{ex:octagon}
Before beginning the proofs in the next section, we give an example of a generic polygon. Consider the following cylindrical growth diagram of type $(\omega_1,\omega_2,\omega_1,\omega_2,\omega_1,\omega_2,\omega_1,\omega_2)$ and $n=8$.

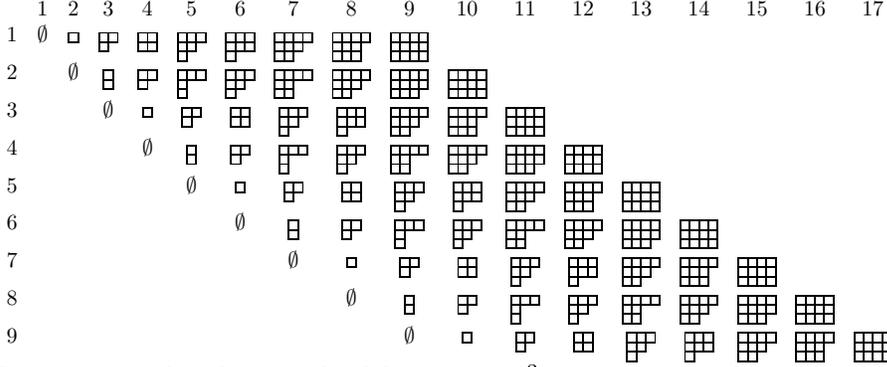
\begin{minipage}{\linewidth}
\scalemath{.75}{
\ytableausetup{boxsize=1.5mm}
\begin{array}{cccccccccccccccccc}
&1&2&3&4&5&6&7&8&9&10&11&12&13&14&15&16&17\\
1&\emptyset&\ydiagram{1, 0, 0}& \ydiagram{2, 1, 0}& \ydiagram{2, 2, 0}& \ydiagram{3, 2, 1}& \ydiagram{3, 3, 1}& \ydiagram{4, 3, 2}& \ydiagram{4, 3, 3}& \ydiagram{4, 4, 4}\\
2&&\emptyset& \ydiagram{1, 1, 0}& \ydiagram{2, 1, 0}& \ydiagram{3, 1, 1}& \ydiagram{3, 2, 1}& \ydiagram{4, 2, 2}& \ydiagram{4, 3, 2}& \ydiagram{4, 4, 3}& \ydiagram{4, 4, 4}\\
3&&&\emptyset& \ydiagram{1, 0, 0}& \ydiagram{2, 1, 0}& \ydiagram{2, 2, 0}& \ydiagram{3, 2, 1}& \ydiagram{3, 3, 1}& \ydiagram{4, 3, 2}& \ydiagram{4, 3, 3}& \ydiagram{4, 4, 4}\\
4&&&&\emptyset& \ydiagram{1, 1, 0}& \ydiagram{2, 1, 0}& \ydiagram{3, 1, 1}& \ydiagram{3, 2, 1}& \ydiagram{4, 2, 2}& \ydiagram{4, 3, 2}& \ydiagram{4, 4, 3}& \ydiagram{4, 4, 4}\\
5&&&&&\emptyset& \ydiagram{1, 0, 0}& \ydiagram{2, 1, 0}& \ydiagram{2, 2, 0}& \ydiagram{3, 2, 1}& \ydiagram{3, 3, 1}& \ydiagram{4, 3, 2}& \ydiagram{4, 3, 3}& \ydiagram{4, 4, 4}\\
6&&&&&&\emptyset& \ydiagram{1, 1, 0}& \ydiagram{2, 1, 0}& \ydiagram{3, 1, 1}& \ydiagram{3, 2, 1}& \ydiagram{4, 2, 2}& \ydiagram{4, 3, 2}& \ydiagram{4, 4, 3}& \ydiagram{4, 4, 4}\\
7&&&&&&&\emptyset& \ydiagram{1, 0, 0}& \ydiagram{2, 1, 0}& \ydiagram{2, 2, 0}& \ydiagram{3, 2, 1}& \ydiagram{3, 3, 1}& \ydiagram{4, 3, 2}& \ydiagram{4, 3, 3}& \ydiagram{4, 4, 4}\\
8&&&&&&&&\emptyset& \ydiagram{1, 1, 0}& \ydiagram{2, 1, 0}& \ydiagram{3, 1, 1}& \ydiagram{3, 2, 1}& \ydiagram{4, 2, 2}& \ydiagram{4, 3, 2}& \ydiagram{4, 4, 3}& \ydiagram{4, 4, 4}\\
9&&&&&&&&&\emptyset& \ydiagram{1, 0, 0}& \ydiagram{2, 1, 0}& \ydiagram{2, 2, 0}& \ydiagram{3, 2, 1}& \ydiagram{3, 3, 1}& \ydiagram{4, 3, 2}& \ydiagram{4, 3, 3}& \ydiagram{4, 4, 4}
\end{array}}
\end{minipage}

Let $e_1,e_2,e_3$ be the standard basis for $\mathcal K^3$. A generic polygon in the component corresponding to this diagram consists of the following eight lattice classes where we list an $\mathcal O$-basis for each lattice representative. Note that the corresponding vertices in the building are not contained in a common apartment. 
\begin{align*}
L_1&=e_1,e_2,e_3 & L_5&=t^{-1}e_1,t^{-2}e_2,e_3\\
L_2&=t^{-1}e_1,e_2,e_3 & L_6&=t^{-1}e_1,t^{-2}e_2,t^{-2}e_1+t^{-1}e_3\\
L_3&=t^{-2}e_1,t^{-1}e_2,e_3 & L_7&=t^{-1}e_1,t^{-1}e_2,t^{-2}e_1+t^{-2}e_2+t^{-1}e_3\\
L_4&=t^{-2}e_1,t^{-2}e_2,e_3 & L_8&=t^{-1}e_1+t^{-1}e_2,e_2,e_3
\end{align*}

By direct computation $\maxconv(P)$ contains the additional lattice classes corresponding to the lattices with $\mathcal O$-bases $(e_1,t^{-1}e_2,e_3)$, $(t^{-2}e_1,t^{-1}e_2,t^{-2}e_2+t^{-1}e_3)$, and $(t^{-1}e_1,t^{-1}e_2,e_3)$. The vertices in $\minconv(P)$, other than those in $P$, have lattice representatives $(t^{-1}e_1,t^{-1}e_2,t^{-2}e_1+t^{-1}e_3)$, $(t^{-2}e_1+t^{-2}e_2,t^{-1}e_2,e_3)$ and $(t^{-1}e_1,t^{-1}e_2,e_3)$. Note that the lattice class with representative $(t^{-1}e_1,t^{-1}e_2,e_3)$ is contained in both. It is the middle vertex of the octagon in the triangulation of $P$ given by $\conv(P)$ shown below.
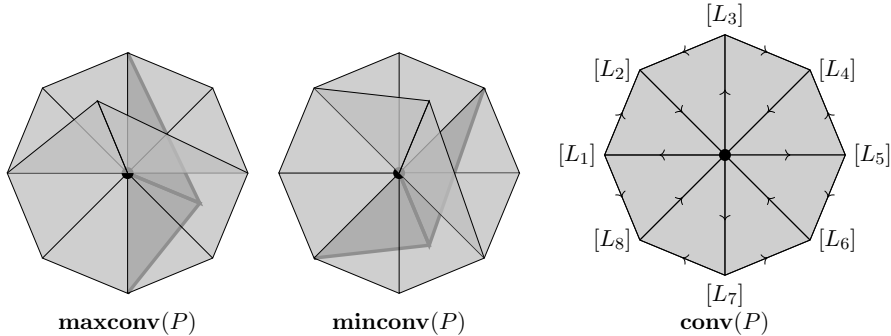
\begin{align*}
\scalemath{.8}{
\begin{tikzpicture}
\fill[gray] (0,-2) -- (1.2,-.5) -- (0,0);
\fill[gray] (0,2) -- (1.2,-.5) -- (0,0);
\draw[ultra thick] (0,-2) -- (1.2,-.5) -- (0,0);
\draw[ultra thick] (0,2) -- (1.2,-.5) -- (0,0);
\foreach \t in {0,45,...,315} {
    \fill[lightgray, nearly opaque] (0,0) -- (\t:2) -- (\t+45:2); 
    \draw (0,0) -- (\t:2) -- (\t+45:2);};
\node at (0,0) [circle, fill, inner sep=2pt] {};
\fill[lightgray, nearly opaque] (-2,0) -- (-.5,1.2) -- (0,0);
\fill[lightgray, nearly opaque] (2,0) -- (-.5,1.2) -- (0,0);
\draw (-2,0) -- (-.5,1.2) -- (0,0);
\draw (2,0) -- (-.5,1.2) -- (0,0);
\node at (0,-2.5) {$\maxconv(P)$};
\end{tikzpicture}
\hspace{5mm}
\begin{tikzpicture}
\fill[gray] (225:2) -- (.5,-1.2) -- (0,0);
\fill[gray] (45:2) -- (.5,-1.2) -- (0,0);
\draw[ultra thick] (225:2) -- (.5,-1.2) -- (0,0);
\draw[ultra thick] (45:2) -- (.5,-1.2) -- (0,0);
\foreach \t in {0,45,...,315} {
    \fill[lightgray, nearly opaque] (0,0) -- (\t:2) -- (\t+45:2); 
    \draw (0,0) -- (\t:2) -- (\t+45:2);};
\node at (0,0) [circle, fill, inner sep=2pt] {};
\fill[lightgray, nearly opaque] (135:2) -- (.5,1.2) -- (0,0);
\fill[lightgray, nearly opaque] (-45:2) -- (.5,1.2) -- (0,0);
\draw (135:2) -- (.5,1.2) -- (0,0);
\draw (-45:2) -- (.5,1.2) -- (0,0);
\node at (0,-2.5) {$\minconv(P)$};
\end{tikzpicture}
\hspace{5mm}
\begin{tikzpicture}
\foreach \t in {0,45,...,315} {
    \fill[lightgray, nearly opaque] (0,0) -- (\t:2) -- (\t+45:2); 
    \draw (0,0) -- (\t:2) -- (\t+45:2);};
\foreach \t in {0,90,...,280}{
	\draw[arrowfrom] (\t:2) -- (0,0);
	\draw[arrowto] (\t:2) -- (\t+45:2);
	\draw[arrowto] (\t:2) -- (\t-45:2);
	\draw[arrowto] (\t+45:2) -- (0,0);
	};
\node at (0,0) [circle, fill, inner sep=2pt] {};
\node at (0,-2.8) {$\conv(P)$};
\node[left] at (180:2) {$[L_1]$};
\node[left] at (135:2) {$[L_2]$};
\node[above] at (90:2) {$[L_3]$};
\node[right] at (45:2) {$[L_4]$};
\node[right] at (0:2) {$[L_5]$};
\node[right] at (-45:2) {$[L_6]$};
\node[below] at (-90:2) {$[L_7]$};
\node[left] at (-135:2) {$[L_8]$};
\end{tikzpicture}}
\end{align*}
\end{Ex}

\section{Proof of the Main Theorem}\label{sec:proof}
The proof of the main theorem will be by induction on $n$, the size of the generic polygon $P$ in $\Delta_2$. The idea is to make local changes to $P$, while keeping track of the set $\conv$. Two of the local changes immediately decrease the number of vertices, which occur in the following situations. We say that a polygon in $\Delta_2$ has a \emph{U-turn configuration} at three consecutive vertices $[L_1], [L_2], [L_3]$ if $[L_1]=[L_3]$. We say that a polygon has a \emph{sharp-corner configuration} at three consecutive vertices $[L_1], [L_2], [L_3]$ if $[L_1]$ and $[L_3]$ are adjacent vertices in $\Delta_2$. We collect in the following propositions the effects of these local changes. Each proposition states that altering the polygon preserves the growth diagram condition and that $\conv(P)$ changes in an expected way.

\begin{figure}
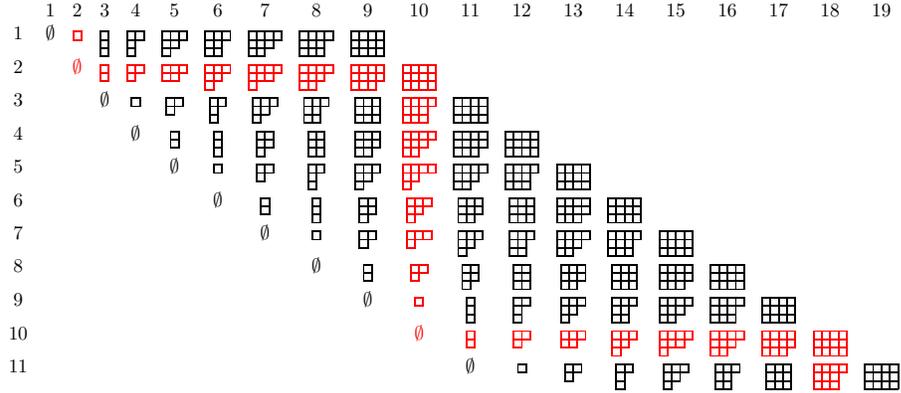

\ytableausetup{boxsize=1.5mm}
\centerline{
  \begin{minipage}{\linewidth}
\begin{align*}
\scalemath{.665}{
\begin{array}{cccccccccccccccccccc}
&1&2&3&4&5&6&7&8&9&10&11&12&13&14&15&16&17&18&19\\
1&\emptyset&\color{red} \ydiagram{1, 0, 0}& \ydiagram{1, 1, 1}& \ydiagram{2, 1, 1}& \ydiagram{3, 2, 1}& \ydiagram{3, 2, 2}& \ydiagram{4, 3, 2}& \ydiagram{4, 3, 3}& \ydiagram{4, 4, 4}\\
2&&\color{red}\emptyset& \color{red}\ydiagram{1, 1, 0}&\color{red} \ydiagram{2, 1, 0}&\color{red} \ydiagram{3, 2, 0}& \color{red}\ydiagram{3, 2, 1}& \color{red}\ydiagram{4, 3, 1}&\color{red} \ydiagram{4, 3, 2}&\color{red} \ydiagram{4, 4, 3}&\color{red} \ydiagram{4, 4, 4}\\
3&&&\emptyset& \ydiagram{1, 0, 0}& \ydiagram{2, 1, 0}& \ydiagram{2, 1, 1}& \ydiagram{3, 2, 1}& \ydiagram{3, 2, 2}& \ydiagram{3, 3, 3}& \color{red}\ydiagram{4, 3, 3}& \ydiagram{4, 4, 4}\\
4&&&&\emptyset& \ydiagram{1, 1, 0}& \ydiagram{1, 1, 1}& \ydiagram{2, 2, 1}& \ydiagram{2, 2, 2}& \ydiagram{3, 3, 2}& \color{red}\ydiagram{4, 3, 2}& \ydiagram{4, 4, 3}& \ydiagram{4, 4, 4}\\
5&&&&&\emptyset& \ydiagram{1, 0, 0}& \ydiagram{2, 1, 0}& \ydiagram{2, 1, 1}& \ydiagram{3, 2, 1}& \color{red}\ydiagram{4, 2, 1}& \ydiagram{4, 3, 2}& \ydiagram{4, 3, 3}& \ydiagram{4, 4, 4}\\
6&&&&&&\emptyset& \ydiagram{1, 1, 0}& \ydiagram{1, 1, 1}& \ydiagram{2, 2, 1}& \color{red}\ydiagram{3, 2, 1}& \ydiagram{3, 3, 2}& \ydiagram{3, 3, 3}& \ydiagram{4, 4, 3}& \ydiagram{4, 4, 4}\\
7&&&&&&&\emptyset& \ydiagram{1, 0, 0}& \ydiagram{2, 1, 0}&\color{red} \ydiagram{3, 1, 0}& \ydiagram{3, 2, 1}& \ydiagram{3, 2, 2}& \ydiagram{4, 3, 2}& \ydiagram{4, 3, 3}& \ydiagram{4, 4, 4}\\
8&&&&&&&&\emptyset& \ydiagram{1, 1, 0}&\color{red} \ydiagram{2, 1, 0}& \ydiagram{2, 2, 1}& \ydiagram{2, 2, 2}& \ydiagram{3, 3, 2}& \ydiagram{3, 3, 3}& \ydiagram{4, 4, 3}& \ydiagram{4, 4, 4}\\
9&&&&&&&&&\emptyset&\color{red} \ydiagram{1, 0, 0}& \ydiagram{1, 1, 1}& \ydiagram{2, 1, 1}& \ydiagram{3, 2, 1}& \ydiagram{3, 2, 2}& \ydiagram{4, 3, 2}& \ydiagram{4, 3, 3}& \ydiagram{4, 4, 4}\\
10&&&&&&&&&&\color{red}\emptyset&\color{red} \ydiagram{1, 1, 0}&\color{red} \ydiagram{2, 1, 0}&\color{red} \ydiagram{3, 2, 0}&\color{red} \ydiagram{3, 2, 1}&\color{red} \ydiagram{4, 3, 1}&\color{red} \ydiagram{4, 3, 2}&\color{red} \ydiagram{4, 4, 3}&\color{red} \ydiagram{4, 4, 4}\\
11&&&&&&&&&&&\emptyset& \ydiagram{1, 0, 0}& \ydiagram{2, 1, 0}& \ydiagram{2, 1, 1}& \ydiagram{3, 2, 1}& \ydiagram{3, 2, 2}& \ydiagram{3, 3, 3}&\color{red} \ydiagram{4, 3, 3}& \ydiagram{4, 4, 4}
\end{array}}
\end{align*}
\end{minipage}}
\caption{A growth diagram for a polygon containing a U-turn. Partitions representing distances involving $[L_2]$ are shown in red.}
\label{U-turn}
\end{figure}
\begin{figure}
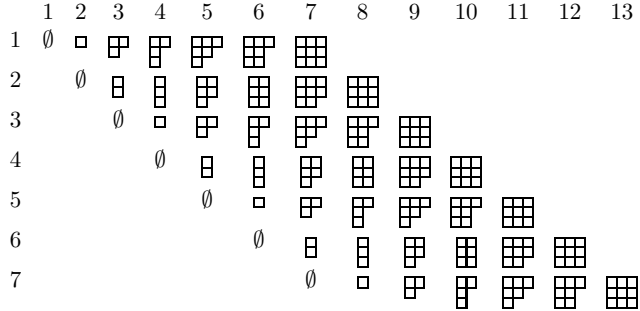

\ytableausetup{boxsize=1.5mm}
\centerline{
  \begin{minipage}{\linewidth}
\begin{align*}
\scalemath{.8}{
\begin{array}{cccccccccccccccccccc}
&1&2&3&4&5&6&7&8&9&10&11&12&13\\
1&\emptyset& \ydiagram{1, 0, 0}& \ydiagram{2, 1, 0}& \ydiagram{2, 1, 1}& \ydiagram{3, 2, 1}& \ydiagram{3, 2, 2}& \ydiagram{3, 3, 3} \\
2&&\emptyset& \ydiagram{1, 1, 0}& \ydiagram{1, 1, 1}& \ydiagram{2, 2, 1}& \ydiagram{2, 2, 2}& \ydiagram{3, 3, 2}&   \ydiagram{3, 3, 3}\\
3&&&\emptyset& \ydiagram{1, 0, 0}& \ydiagram{2, 1, 0}& \ydiagram{2, 1, 1}& \ydiagram{3, 2, 1}&  \ydiagram{3, 2, 2}& \ydiagram{3, 3, 3}\\
4&&&&\emptyset& \ydiagram{1, 1, 0}& \ydiagram{1, 1, 1}& \ydiagram{2, 2, 1}&   \ydiagram{2, 2, 2}& \ydiagram{3, 3, 2}& \ydiagram{3, 3,3}\\
5&&&&&\emptyset& \ydiagram{1, 0, 0}& \ydiagram{2, 1, 0}& \ydiagram{2, 1, 1}& \ydiagram{3, 2, 1}& \ydiagram{3, 2, 2}& \ydiagram{3, 3, 3}\\
6&&&&&&\emptyset& \ydiagram{1, 1, 0}&  \ydiagram{1, 1, 1}& \ydiagram{2, 2, 1}& \ydiagram{2, 2, 2}& \ydiagram{3, 3, 2}& \ydiagram{3, 3, 3}\\
7&&&&&&&\emptyset& \ydiagram{1, 0, 0}& \ydiagram{2, 1, 0}& \ydiagram{2, 1, 1}& \ydiagram{3, 2, 1}& \ydiagram{3, 2, 2}& \ydiagram{3, 3, 3} \\
\end{array}}
\end{align*}
\end{minipage}}
\caption{The growth diagram resulting from the diagram in Figure \ref{U-turn} with the red distances removed and corresponding entries identified as in the proof of Proposition \ref{prop:U-turn}.}
\label{U-turn removed}
\end{figure}

\begin{Prop}\label{prop:U-turn}
Let $P$ be a generic polygon in $\Delta_2$. If $P$ contains a U-turn at $[L_1],[L_2],[L_3]$, then the polygon with $[L_2]$ removed and $[L_1]$ identified with $[L_3]$ is a generic polygon. Furthermore, $\conv(P)=\conv(P\setminus [L_2])\cup \{[L_2]\}$.
\end{Prop}
\begin{proof}
Let $\gamma_{ij}$ be the partitions representing the pairwise distances $d([L_i],[L_j])$ in the growth diagram for $P$. Suppose that $\gamma_{12}=\omega_1$, $\gamma_{23}=\omega_2$, and $\gamma_{13}=\omega_3$, the other case being similar. See Figure $\ref{U-turn}$ for an example of such a growth diagram. We must show that removing the vertex $[L_2]$ and identifying $[L_1]$ with $[L_3]$ yields a generic polygon - that the corresponding pairwise distances yield a valid cylindrical growth diagram. Removing the vertex $[L_2]$ corresponds to removing the row entries $\gamma_{2,j}$ and column entries $\gamma_{i,2}$ from the growth diagram. Identifying the vertices $[L_1]$ and $[L_3]$ does not change any of the distances $d([L_i],[L_j])$. Since $[L_1]=[L_3]$, we have $d([L_1],[L_j])=d([L_3],[L_j])$ for all valid $j$, so that as partitions $\gamma_{1,j}$ and $\gamma_{3,j}$ differ by a full column. More specifically, $\gamma_{1,j}=\gamma_{3,j}+(1,1,1)$ for all valid $j$. Similarly, $\gamma_{i,n+1}+(1,1,1)=\gamma_{i,n+3}$ for all valid $i$.

Consider the triangle of partitions $\gamma_{i,j}$ for indices $3\leq i\leq n+1$ and $n+3\leq j\leq n+i$. Remove a full column  $(1,1,1)$ from each partition in this triangle to get the partitions $\gamma_{i,j}'$. This allows us to merge the columns $\gamma_{i,n+1}=\gamma_{i,n+3}'$. Similarly, merge the first and third rows. See Figure \ref{U-turn} and the resulting diagram in Figure \ref{U-turn removed}. The new diagram satisfies the local condition in every unit square because the unit square came from the original diagram, or it came from one in the original diagram whose partitions all decreased by a full column.

Now we show that $\conv(P)=\conv(P\setminus [L_2])\cup \{[L_2]\}$. Since $\dif(\gamma_{13},\gamma_{23})=\{3\}$, Proposition \ref{prop:difs} implies that $\dif(\gamma_{1,j},\gamma_{2,j})=\{3\}$ for all $j\geq 3$. Hence, for all $j$ we have $\lvert \gamma_{2,j}\rvert=\lvert \gamma_{1,j}\rvert+1$. That is, for all $j$ the vertex $[L_1]=[L_3]$ is strictly closer to $[L_j]$ than $[L_2]$. Placing the edge $([L_1],[L_2])$ in a common apartment with $[L_j]$, we see that any nontrivial $\conv([L_2],[L_j])$ is equal to $\conv([L_1],[L_j])\cup \{[L_2]\}$. Together with Proposition \ref{prop:conv} the result follows.
\end{proof}

\begin{figure}[b]
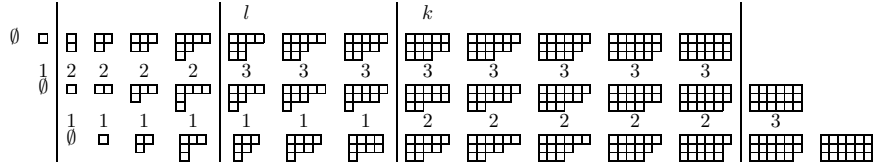

\centerline{
  \begin{minipage}{\linewidth}
\begin{align*}
\scalemath{.69}{
\ytableausetup{boxsize=1.5mm}
\begin{array}{ccc|cccc|ccc|ccccc|cc}
&&&&&&&l&&&k&&&&&\\
&\emptyset&\ydiagram{1}&\ydiagram{1,1}&\ydiagram{2,1}&\ydiagram{3,2}&\ydiagram{4,2,1}&\ydiagram{4,2,2}&\ydiagram{5,3,2}&\ydiagram{5,4,2}&\ydiagram{5,5,3}&\ydiagram{6,5,3}&\ydiagram{6,5,4}&\ydiagram{6,6,5}&\ydiagram{6,6,6}&\\
&&1&2&2&2&2&3&3&3&3&3&3&3&3\\[-1.5mm]
&&\emptyset&\ydiagram{1}&\ydiagram{2}&\ydiagram{3,1}&\ydiagram{4,1,1}&\ydiagram{4,2,1}&\ydiagram{5,3,1}&\ydiagram{5,4,1}&\ydiagram{5,5,2}&\ydiagram{6,5,2}&\ydiagram{6,5,3}&\ydiagram{6,6,4}&\ydiagram{6,6,5}&\ydiagram{6,6,6}&\\
&&&1&1&1&1&1&1&1&2&2&2&2&2&3\\[-1.5mm]
&&&\emptyset&\ydiagram{1}&\ydiagram{2,1}&\ydiagram{3,1,1}&\ydiagram{3,2,1}&\ydiagram{4,3,1}&\ydiagram{4,4,1}&\ydiagram{5,4,2}&\ydiagram{6,4,2}&\ydiagram{6,4,3}&\ydiagram{6,5,4}&\ydiagram{6,5,5}&\ydiagram{6,6,5}&\ydiagram{6,6,6}\\
\end{array}}
\end{align*}
\end{minipage}}
\caption{First three rows of a growth diagram for a polygon that contains a sharp corner. The column indices $l$ and $k$ from the proof of Proposition \ref{prop:sharp corner} are indicated. \label{fig:sharp corner}}
\label{sharp}
\end{figure}
\begin{figure}
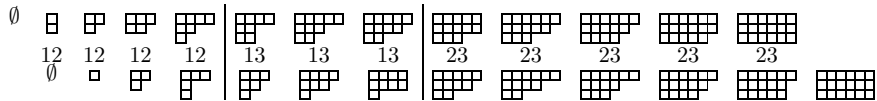

\centerline{
  \begin{minipage}{\linewidth}
\begin{align*}
\scalemath{.77}{
\ytableausetup{boxsize=1.5mm}
\begin{array}{ccccc|ccc|cccccccc}
\emptyset&\ydiagram{1,1}&\ydiagram{2,1}&\ydiagram{3,2}&\ydiagram{4,2,1}&\ydiagram{4,2,2}&\ydiagram{5,3,2}&\ydiagram{5,4,2}&\ydiagram{5,5,3}&\ydiagram{6,5,3}&\ydiagram{6,5,4}&\ydiagram{6,6,5}&\ydiagram{6,6,6}&\\
&12&12&12&12&13&13&13&23&23&23&23&23&&\\[-1.5mm]
&\emptyset&\ydiagram{1}&\ydiagram{2,1}&\ydiagram{3,1,1}&\ydiagram{3,2,1}&\ydiagram{4,3,1}&\ydiagram{4,4,1}&\ydiagram{5,4,2}&\ydiagram{6,4,2}&\ydiagram{6,4,3}&\ydiagram{6,5,4}&\ydiagram{6,5,5}&\ydiagram{6,6,6}\\
\end{array}}
\end{align*}
\end{minipage}}
\caption{The growth diagram from Figure \ref{sharp} with distances corresponding to $[L_2]$ removed.}
\label{sharp removed}
\end{figure}

\begin{Prop}\label{prop:sharp corner}
Let $P$ be a generic polygon in $\Delta_2$. If $P$ contains a sharp corner at $[L_1],[L_2],[L_3]$, then $P$ with $[L_2]$ removed and $[L_1]$ connected to $[L_3]$ is a generic polygon. Furthermore, $\conv(P)=\conv(P\setminus [L_2])\cup \{[L_2]\}$. 
\end{Prop}
\begin{proof}
As in the previous proof, let $\gamma_{ij}$ be the partitions representing the pairwise distances $d([L_i],[L_j])$ in the growth diagram for $P$. Since there is a sharp corner, suppose that $\gamma_{1,2}=\omega_1$, $\gamma_{2,3}=\omega_1$, and $\gamma_{1,3}=\omega_2$. The other case with $\gamma_{1,2}=\omega_2$, $\gamma_{2,3}=\omega_2$, and $\gamma_{1,3}=\omega_1$ is similar. See Figure \ref{fig:sharp corner} for an example of a growth diagram for this setup. We must show that removing the entries corresponding to distances involving $[L_2]$ yields a valid growth diagram since the distances $d([L_i],[L_j])$ for $i\not=2$ and $j\not=2$ are not affected when $[L_2]$ is removed. Unlike the previous proposition, since the rows and columns corresponding to $[L_1]$ and $[L_3]$ are not identified, we must check that $\gamma_{1,i}$ and $\gamma_{3,i}$ differ in a vertical strip and that the new unit squares satisfy the local condition. 

The differences in the third column across the first three rows are $\dif(\gamma_{1,3},\gamma_{2,3})=\{2\}$ and $\dif(\gamma_{2,3},\gamma_{3,3})=\{1\}$. Let $l$ be the smallest index such that $\dif(\gamma_{1,l},\gamma_{2,l})=\{3\}$. This is the smallest index such $\gamma_{1,l}=(a,b,b)$ for integers $a\geq b\geq 0$. Similarly, define $k$ to be the smallest index such that $\dif(\gamma_{2,k},\gamma_{3,k})\not=\{1\}$. This is the smallest index such that $\gamma_{2,k}=(c,c,d)$ for some integers $c\geq d\geq 0$. Note that $l\leq k$ because $\gamma_{1,j}$ and $\gamma_{2,j}$ differ by a box in the second row for $j<l$, so $\gamma_{2,k}$ cannot possibly be of the form $(c,c,d)$. Note that $\dif(\gamma_{2,k},\gamma_{3,k})=\{2\}$, i.e. the differences $\dif(\gamma_{2,j},\gamma_{3,j})$ across the second and third rows cannot switch directly from $\{1\}$ to $\{3\}$ as the column index $j$ increases. Such a switch would require $\gamma_{2,j}$ to be of the form $(a,a,a)$ for some $j$, which is impossible since $\dif(\gamma_{1,j},\gamma_{2,j})=\{2\}$ or $\{3\}$ for all $j>2$. Hence, the smallest index $j$ for which $\dif(\gamma_{2,j},\gamma_{3,j})=\{3\}$ is $j=n+2$, the index of the column containing the last non-empty entry of the second row. Hence, there are $3$ portions of the diagram to analyze if $l\not=k$ and $2$ otherwise (see Figure \ref{fig:sharp corner}).

In either case, for all $j<l$, $\gamma_{1,j}$ and $\gamma_{3,j}$ differ by a vertical strip of size two, more specifically $\dif(\gamma_{1,i},\gamma_{3,i})=\{1,2\}$. Hence, removing the entries $\gamma_{2,j}$ from the diagram for $j<l$ yields a portion of a new diagram that satisfies the local rule in each new unit square. 

Suppose that $l\not=k$. Then for each $j$ such that $l\leq j<k$, $\dif(\gamma_{1,j},\gamma_{3,j})=\{1,3\}$. Hence, the local condition is satisfied for unit squares $\gamma_{1,j-1},\gamma_{1,j},\gamma_{3,j-1},\gamma_{3,j}$ for all $j$ in the range $l<j<k$. We must also check the local condition for the unit square involving the four entries $\gamma_{1,l-1},\gamma_{1,l},\gamma_{3,l-1},\gamma_{3,l}$. As mentioned above, $\gamma_{1,l}=(a,b,b)$ for integers $a,b$. The integers $a,b$ are not equal because $a=b$ would imply $l=k$. The weight $\gamma_{3,l}$ is $(a-1,b,b-1)$, which we show agrees with the local condition. Subtracting one from the entries $\dif(\gamma_{1,l-1},\gamma_{3,l-1})=\{1,2\}$ of $\gamma_{1,l}$ yields the weight $(a-1,b-1,b)$. According to the local condition, and since $a\not=b$, this is then sorted to get $\gamma_{3,l}=(a-1,b,b-1)$ as desired.

Now consider the unit square involving $\gamma_{1,k-1},\gamma_{1,k},\gamma_{3,k-1},\gamma_{3,k}$, still assuming $l\not=k$. We have that $\dif(\gamma_{1,k-1},\gamma_{3,k-1})=\{1,3\}$ and $\dif(\gamma_{1,k},\gamma_{3,k})=\{2,3\}$. As previously mentioned $\gamma_{2,k}=(c,c,d)$ where $c\not=d$ because $\dif(\gamma_{1,k},\gamma_{2,k})=\{3\}$. Then $\gamma_{1,k}=(c,c,d+1)$ and $\gamma_{3,k}=(c,c-1,d)$. According to the local condition, subtracting $1$ in entries $1$ and $3$ yields the weight $(c-1,c,d)$, which sorts to $(c,c-1,d)$ as desired.

Now suppose that $l=k$. This is only possible if $\gamma_{1,l}=(a,a,a)$ for some integer $a$. This forces $\gamma_{1,l-1}$ to be $(a,a,a-1)$ ($\gamma_{1,l-1}=(a,a-1,a-1)$ would imply $l<k$), $\gamma_{3,l-1}=(a-1,a-1,a-1)$, and $\gamma_{3,l}=(a,a-1,a-1)$, which satisfy the local rule with $\dif(\gamma_{1,l-1},\gamma_{3,l-1})=\{1,2\}$ and $\dif(\gamma_{1,l},\gamma_{3,l})=\{2,3\}$.  

For $j\geq k$ in either case of $l\not=k$ and $l=k$, $\dif(\gamma_{1,i},\gamma_{3,i})=\{2,3\}$ and the local condition is satisfied for all unit squares involving $\gamma_{1,j}$ and $\gamma_{3,j}$ for $j\geq k$.

To see that $\conv(P)=\conv(P\setminus [L_2])\cup \{L_2\}$, consider first $j<k$. Since $\dif(\gamma_{2,j},\gamma_{3,j})=\{1\}$ and the first two parts of $\gamma_{2,j}$ cannot be equal for all $j<k$, $\lvert \gamma_{2,j}\rvert=\lvert \gamma_{3,j}\rvert+1$ for all $j<k$. Hence, for any non-trivial $\conv([L_2],[L_j])$, it holds that $\conv([L_2],[L_j])=\conv([L_3],[L_j])\cup \{[L_2]\}$. Likewise, for $j\geq k$, $\dif(\gamma_{1,j},\gamma_{2,j})=\{3\}$, so $\lvert \gamma_{2,j}\rvert=\lvert \gamma_{1,j}\rvert+1$ and $\conv([L_2],[L_j])=\conv([L_1],[L_j])\cup \{[L_2]\}$ for any non-trivial $\conv([L_2],[L_j])$.
\end{proof}

The previous two propositions allow for the elimination of U-turn configurations and sharp-corner configurations from polygons. We will say that three consecutive vertices  $[L_1],[L_2],[L_3]$ in a path or polygon form an \emph{elbow} if they are all distinct and $d([L_1],[L_2])\not=d([L_2],[L_3])$. The following lemma is an exercise. 

\begin{Lem}\label{lem:double elbow}
Let $P$ be a polygon in $\Delta_2$ that does not have any U-turn configurations nor any sharp-corner configurations. Then there is a sequence of consecutive vertices $[L_1],\ldots,[L_a]$ of $P$ such that $[L_1],[L_2],[L_3]$ and $[L_{a-2}],[L_{a-1}],[L_a]$ are both elbow configurations, and $a$ is the smallest index such that $d([L_{j-1}],[L_{j}])=d([L_j],[L_{j+1}])$ for all $3\leq j\leq a-2$ and $\lvert d([L_1],[L_a]) \rvert=\lvert d([L_1],[L_{a-1}]) \rvert$. The sequence of vertices $[L_1],\ldots,[L_a]$ can be put in a common apartment.
\end{Lem}

For example, if $d([L_1],[L_2])=\omega_1$ and $d([L_2],[L_3])=\omega_2$, then $d([L_1],[L_j])=\omega_1+(j-2)\omega_2$ for $2\leq j<a$ and $d([L_1],[L_a])=(a-2)\omega_2$. See Figure \ref{fig:double elbow}. We will call the path $[L_1],\ldots,[L_a]$ a \emph{double-elbow configuration}.

\begin{figure}[b]
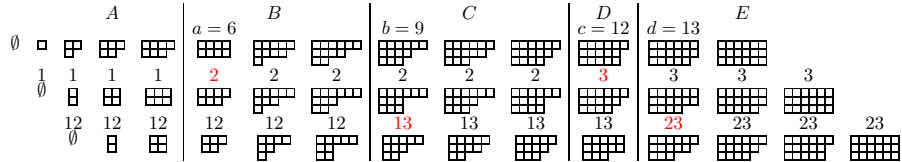

\centerline{
  \begin{minipage}{\linewidth}
\begin{align*}
\scalemath{.65}{
\ytableausetup{boxsize=1.5mm}
\begin{array}{ccccc|ccc|ccc|c|ccccc}
&&&A&&&B&&&C&&D&&E&\\[-1.5mm]
&&&&&a=6&&&b=9&&&c=12&d=13&&\\[-1.5mm]
\emptyset&\ydiagram{1}&\ydiagram{2,1}&\ydiagram{3,2}&\ydiagram{4,3}&\ydiagram{4,4}&\ydiagram{5,4,1}&\ydiagram{6,4,2}&\ydiagram{6,4,3}&\ydiagram{6,5,3}&\ydiagram{6,5,4}&\ydiagram{6,5,5}&\ydiagram{6,6,5}&\ydiagram{6,6,6}\\
&1&1&1&1&\color{red} 2&2&2&2&2&2&\color{red}3&3&3&3\\[-1.5mm]
&\emptyset&\ydiagram{1,1}&\ydiagram{2,2}&\ydiagram{3,3}&\ydiagram{4,3}&\ydiagram{5,3,1}&\ydiagram{6,3,2}&\ydiagram{6,3,3}&\ydiagram{6,4,3}&\ydiagram{6,4,4}&\ydiagram{6,5,4}&\ydiagram{6,6,4}&\ydiagram{6,6,5}&\ydiagram{6,6,6}\\
&&12&12&12&12&12&12&\color{red}13&13&13&13&\color{red}23&23&23&23\\[-1.5mm]
&&\emptyset&\ydiagram{1,1}&\ydiagram{2,2}&\ydiagram{3,2}&\ydiagram{4,2,1}&\ydiagram{5,2,2}&\ydiagram{5,3,2}&\ydiagram{5,4,2}&\ydiagram{5,4,3}&\ydiagram{5,5,3}&\ydiagram{6,5,3}&\ydiagram{6,5,4}&\ydiagram{6,5,5}&\ydiagram{6,6,6}\\
\end{array}}
\end{align*}
\end{minipage}}
\caption{First three rows of a growth diagram containing a double-elbow configuration. The column indices $a,b,c,d$ from the proof of Proposition \ref{prop:elbow move} are shown for this example. \label{fig:elbow configuration}}
\label{elbow}
\end{figure}
\begin{figure}
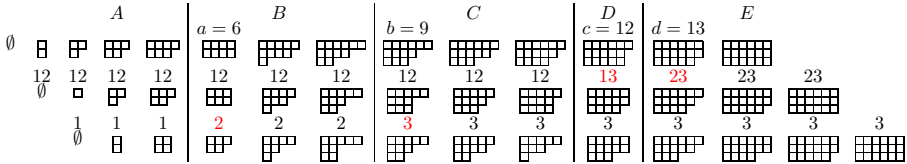

\centerline{
  \begin{minipage}{\linewidth}
\begin{align*}
\scalemath{.65}{
\ytableausetup{boxsize=1.5mm}
\begin{array}{ccccc|ccc|ccc|c|ccccc}
&&&A&&&B&&&C&&D&&E&\\[-1.5mm]
&&&&&a=6&&&b=9&&&c=12&d=13&&\\[-1.5mm]
\emptyset&\ydiagram{1,1}&\ydiagram{2,1}&\ydiagram{3,2}&\ydiagram{4,3}&\ydiagram{4,4}&\ydiagram{5,4,1}&\ydiagram{6,4,2}&\ydiagram{6,4,3}&\ydiagram{6,5,3}&\ydiagram{6,5,4}&\ydiagram{6,5,5}&\ydiagram{6,6,5}&\ydiagram{6,6,6}\\
&12&12&12&12&12&12&12&12&12&12&\color{red}13&\color{red}23&23&23\\[-1.5mm]
&\emptyset&\ydiagram{1}&\ydiagram{2,1}&\ydiagram{3,2}&\ydiagram{3,3}&\ydiagram{4,3,1}&\ydiagram{5,3,2}&\ydiagram{5,3,3}&\ydiagram{5,4,3}&\ydiagram{5,4,4}&\ydiagram{5,5,4}&\ydiagram{6,5,4}&\ydiagram{6,6,5}&\ydiagram{6,6,6}\\
&&1&1&1&\color{red}2&2&2&\color{red}3&3&3&3&3&3&3&3\\[-1.5mm]
&&\emptyset&\ydiagram{1,1}&\ydiagram{2,2}&\ydiagram{3,2}&\ydiagram{4,2,1}&\ydiagram{5,2,2}&\ydiagram{5,3,2}&\ydiagram{5,4,2}&\ydiagram{5,3,1}&\ydiagram{5,5,3}&\ydiagram{6,5,3}&\ydiagram{6,6,4}&\ydiagram{6,6,5}&\ydiagram{6,6,6}\\
\end{array}}
\end{align*}
\end{minipage}}
\caption{Example from Figure \ref{elbow} with $[L_2]$ replaced by $[L_2']$. \label{elbow move}}
\label{elbow switched}
\end{figure}

\begin{Prop}\label{prop:elbow move}
Let $P$ be a generic polygon in $\Delta_2$. Suppose $P$ contains a double-elbow configuration at $[L_1],\ldots,[L_a]$. Let $[L_2']$ be the unique vertex adjacent to $[L_1],[L_2],[L_3]$. Then the polygon $P'=\left([L_1],[L_2'],[L_3],\ldots,[L_n]\right)$ is generic. Furthermore, $\conv(P)=\conv(P')\cup \{[L_2]\}$.
\end{Prop}

\begin{proof}
Let $\gamma_{ij}$ be the partitions representing the pairwise distances $d([L_i],[L_j])$ in the growth diagram for $P$. Since there is an elbow configuration starting at $[L_1]$, suppose that $\gamma_{1,2}=\omega_1$, $\gamma_{2,3}=\omega_2$, and $\gamma_{1,3}=\omega_1+\omega_2$. See Figure \ref{fig:elbow configuration} for an example of the first three rows of a growth diagram with this setup. The other case is similar. 

Let $[L_2']$ be the unique vertex adjacent to $[L_1],[L_2]$ and $[L_3]$. We must show that replacing the partitions $\gamma_{2,j}$ and $\gamma_{i,2}$ with partitions representing $d([L_2'],[L_j])$ and $d([L_i],[L_2'])$ yields a valid growth diagram. Note that the only entry that changes in the first row is $\gamma_{1,2}=d([L_1],[L_2])=\omega_1$, which becomes $d([L_1],[L_2'])=\omega_2$, and this new first row determines the entire new growth diagram. We will show that partitions for $d([L_2'],[L_j])$ are precisely the entries of the second row of this new diagram. In what follows, we interpret $\omega_1, \omega_2$ as both weights and the partitions $(1,0,0), (1,1,0)$ respectively. It should be clear from context which is meant.

Suppose that the double-elbow configuration consists of $a$ vertices $[L_1],[L_2],[L_3],\ldots,[L_a]$, so that 
\begin{align*}
\begin{array}{lccl}
\gamma_{1,j}=\omega_1+(j-2)\omega_2& \text{ for } 2\leq j<a, &\;&\gamma_{1,a}=(a-2)\omega_2\\
\gamma_{2,j}=(j-2)\omega_2& \text{ for } 3\leq j<a, &\;&\gamma_{2,a}=\omega_1+(a-3)\omega_2\\
\gamma_{3, j}=(j-3)\omega_2& \text{ for } 4\leq j<a, &\;&\gamma_{3,a}=\omega_1+(a-4)\omega_2
\end{array}
\end{align*}
See Figure \ref{fig:elbow configuration} for an example of the first three rows of a growth diagram with $a=6$. 

For $j$ such that $3\leq j<a$ the distances involving $[L_2']$ are $d(L_2',L_j)=\omega_1+(j-3)\omega_2$. Replace $\gamma_{2,j}$ with $\gamma_{2,j}'=\omega_1+(j-3)\omega_2$ and $\gamma_{1,2}$ with $\omega_2$. Then $\dif(\gamma_{1,j},\gamma_{2,j}')=\{1,2\}$ for all $3\leq j<a$, so the local condition is satisfied in each unit square across the first two rows of the new diagram. The local condition across the second and third rows is also satisfied for all unit squares involving indices $3\leq j<a$ because $\dif(\gamma_{2,j}',\gamma_{3,j})=\{1\}$ for all such $j$. See Figure \ref{elbow move} for the first three rows of the diagram corresponding to the diagram from Figure~\ref{fig:elbow configuration} with $[L_2]$ replaced by $[L_2']$. 

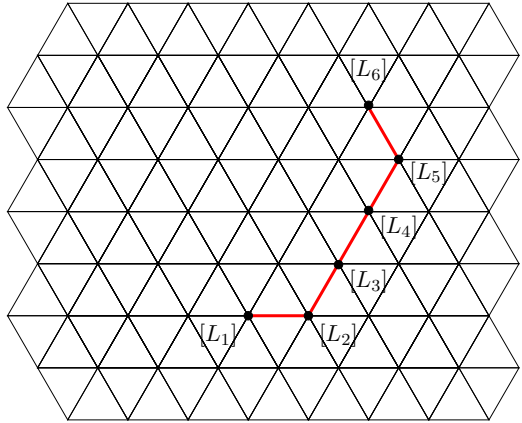
\begin{figure}
\center
\scalemath{.8}{
\begin{tikzpicture}
\foreach \j in {-1,...,5}{
\foreach \i in {-3,...,3}{
	\ifodd\j{
	\draw (\i+1/2,0.866*\j) -- ++(0:1) -- ++(120:1) -- ++(-120:1);
	\draw (\i+1/2,0.866*\j) -- ++(60:1) -- ++(180:1) -- ++(-60:1);
	\draw (\i+1/2,0.866*\j) -- ++(120:1) -- ++(-120:1) -- ++(0:1);
	\draw (\i+1/2,0.866*\j) -- ++(180:1) -- ++(-60:1) -- ++(60:1);
	\draw (\i+1/2,0.866*\j) -- ++(-120:1) -- ++(0:1) -- ++(120:1);
	\draw (\i+1/2,0.866*\j) -- ++(-60:1) -- ++(60:1) -- ++(180:1);
	}
	\else{
	\draw (\i,0.866*\j) -- ++(0:1) -- ++(120:1) -- ++(-120:1);
	\draw (\i,0.866*\j) -- ++(60:1) -- ++(180:1) -- ++(-60:1);
	\draw (\i,0.866*\j) -- ++(120:1) -- ++(-120:1) -- ++(0:1);
	\draw (\i,0.866*\j) -- ++(180:1) -- ++(-60:1) -- ++(60:1);
	\draw (\i,0.866*\j) -- ++(-120:1) -- ++(0:1) -- ++(120:1);
	\draw (\i,0.866*\j) -- ++(-60:1) -- ++(60:1) -- ++(180:1);
	}\fi
}
}
\draw [line width=0.5mm, red] (0,0) -- (1,0) -- (2.5,3*.866) -- (2,4*.866);
\filldraw (0:0) circle(2pt);
\filldraw (1,0) circle(2pt);
\filldraw (1.5,.85) circle(2pt);
\filldraw (2,1.75) circle(2pt);
\filldraw (2.5,2.6) circle(2pt);
\filldraw (2,3.5) circle(2pt);
\node[below] at (-.5,0) {$[L_1]$}; 
\node[below] at (1.5,0) {$[L_2]$};
\node[below] at (2,.9) {$[L_3]$}; 
\node[below] at (2.5,1.8) {$[L_4]$}; 
\node[below] at (3,2.7) {$[L_5]$}; 
\node[above] at (2,3.75) {$[L_6]$};
\end{tikzpicture}}
\caption{A double-elbow configuration. \label{fig:double elbow}}
\end{figure}

Divide $P$ into subpaths A, B, C, D, E, where part A consists of the initial subset $\{[L_1],\ldots,[L_{a-1}]\}$ of $P$. Note that $a$ is the smallest index such that $\dif(\gamma_{1,a},\gamma_{2,a})=\{2\}$. Let $b$ be the smallest index such that $\dif(\gamma_{2,b},\gamma_{3,b})\not=\{1,2\}$. Let $c$ be the smallest index such that $\dif(\gamma_{1,c},\gamma_{2,c})=\{3\}$, and $d$ the smallest index such that $\dif(\gamma_{2,d},\gamma_{3,d})=\{2,3\}$. We claim that $a<b<c\leq d$ and set
\begin{align*}
A&=\{[L_1],\ldots,[L_{a-1}]\}\\
B&=\{[L_a],\ldots,[L_{b-1}]\}\\
C&=\{[L_b],\ldots,[L_{c-1}]\}\\
D&=\{[L_c],\ldots,[L_{d-1}]\}\\
E&=\{[L_d],\ldots,[L_{n}]\}.
\end{align*}

Since $\gamma_{1,a}=(a-2)\omega_2$, $\gamma_{2,a}=\omega_1+(a-3)\omega_2$, and $\gamma_{3,a}=\omega_1+(a-4)\omega_2$, it follows that $\dif(\gamma_{1,a},\gamma_{2,a})=\{2\}$ and $\dif(\gamma_{2,a},\gamma_{3,a})=\{1,2\}$, so $a$ is strictly smaller than $b$. To see that $b$ is strictly less than $c$ note that $c$ is the smallest index such that $\gamma_{1,c}=(x,y,y)$ for integers $x,y$ with $x\geq y\geq 0$. This implies that $\gamma_{1,(c-1)}=(z,y,y-1)$ for $z=x$ or $z=x-1$ (the latter is only possible if $x>y$). Since $\dif(\gamma_{1,c-1},\gamma_{2,c-1})=\{2\}$, we have $\gamma_{2,(c-1)}=(z,y-1,y-1)$. This implies that $\dif(\gamma_{2,(c-1)},\gamma_{3,(c-1)})$ cannot be $\{1,2\}$, so $b<c$. Now consider $j$ such that $a\leq j<c$, i.e. the indices $j$ such that $\dif(\gamma_{1,j},\gamma_{2,j})=\{2\}$. Since $d$ is the smallest index such that $\dif(\gamma_{2,j},\gamma_{3,j})=\{2,3\}$, this requires $\gamma_{2,d}=(x,x,y)$ for some integers $x\geq y\geq 0$. But this is impossible for any $j$ in the range $a\leq j<c$ because $\dif(\gamma_{1,j},\gamma_{2,j})=\{2\}$, so $c\leq d$. Note it is possible that $c=d$. \\ 

For each region we will show two things: 1) the local condition holds and 2) $\conv(P)$ changes as expected, i.e. $\conv([L_2'],[L_j])\subset \conv(P)$ and $\conv([L_2],[L_j])\subset \conv(P')\cup\{[L_2]\}$. Note that both hold for subpath A.\\ 

\noindent\underline{\bf Subpath B.} 

\noindent{\bf Local Condition:}
Consider region B consisting of the vertices $[L_j]$ such that $a\leq j<b$ for which $\dif(\gamma_{1,j},\gamma_{2,j})=\{2\}$ and $\dif(\gamma_{2,j},\gamma_{3,j})=\{1,2\}$. The difference $\dif(\gamma_{2,j},\gamma_{3,j})=\{1,2\}$ implies $\gamma_{2,j}=\gamma_{3,j}+\omega_2$, and similarly $\dif(\gamma_{1,j},\gamma_{2,j})=\{2\}$ implies that $\gamma_{2,j}=\gamma_{1,j}-\omega_2+\omega_1$. In particular $\lvert \gamma_{2,j}\rvert=\lvert \gamma_{3,j}\rvert+1$ and $\lvert \gamma_{2,j}\rvert=\lvert \gamma_{1,j}\rvert$. 

Place the edge $([L_1],[L_2])$ and $[L_j]$ in a common apartment $A$ containing the parallelogram of combinatorial geodesics from $[L_2]$ to $[L_j]$ (see Figure \ref{fig:region B}). Since $[L_3]$ is strictly closer to $[L_j]$ than $[L_2]$ it lies on this parallelogram (one $\omega_2$ step from $[L_2]$). Since any apartment containing $[L_1]$ and $[L_3]$ must contain all of their combinatorial geodesics, $A$ contains $[L_2']$, which has to be the other vertex adjacent to $[L_2]$ in the parallelogram of geodesics from $[L_2]$ to $[L_j]$. This implies that $d(L_2',L_j)=\gamma_{2,j}-\omega_1=\gamma_{1,j}-\omega_2=\gamma_{3,j}+\omega_2-\omega_1$. Let $\gamma_{2,j}'$ be the corresponding partitions. This means that $\dif(\gamma_{1,j},\gamma_{2,j}')=\{1,2\}$ and $\dif(\gamma_{2,j}',\gamma_{3,j})=\{2\}$ for all $a\leq j<b$. Hence, the local condition is satisfied for each unit square $\gamma_{1,j-1}, \gamma_{1,j}, \gamma_{2,j-1}, \gamma_{2,j}$ for $a< j<b$, and for each unit square across the second row, $\gamma_{2,j-1}, \gamma_{2,j}, \gamma_{3,j-1}, \gamma_{3,j}$ for $a< j<b$. 

Now consider the unit squares across columns $a-1,a$. Since $\gamma_{1,a}=(a-2)\omega_2$, $\gamma_{1,a-1}=\omega_1+(a-3)\omega_2$, $d([L_2'],[L_a])=(a-3)\omega_2$, $d([L_2'],[L_{a-1}])=\omega_1+(a-2)\omega_2$ the local condition is satisfied. For the square across the next two rows, $\gamma_{3,a}=\omega_1+(a-4)\omega_2$, $\gamma_{3,a-1}=(a-4)\omega_2$, so the local condition is satisfied. \\ 

\noindent{\bf Effects on conv:}
Note that for all $a\leq j<b$, the distance $d([L_2],[L_j])$ is never of the form $k\omega_1$ or $k\omega_2$ for any $k$, so $\conv([L_2],[L_j])=\{[L_2],[L_j]\}$. On the other hand, $d([L_2'],[L_j])$ is a multiple of $\omega_2$ only when $i=a$, in which case the straight-path geodesic from $[L_1]$ to $[L_a]$ contains the straight-path geodesic from $[L_2']$ to $[L_a]$. Hence, $\conv([L_1],[L_a])=\conv([L_2'],[L_a])\cup \{[L_1]\}$.\\ 

\begin{figure}
\center
\scalemath{.8}{
\begin{tikzpicture}
\foreach \j in {0,...,6}{
\foreach \i in {-3,...,3}{
	\ifodd\j{
	\draw (\i+1/2,0.866*\j) -- ++(0:1) -- ++(120:1) -- ++(-120:1);
	\draw (\i+1/2,0.866*\j) -- ++(60:1) -- ++(180:1) -- ++(-60:1);
	\draw (\i+1/2,0.866*\j) -- ++(120:1) -- ++(-120:1) -- ++(0:1);
	\draw (\i+1/2,0.866*\j) -- ++(180:1) -- ++(-60:1) -- ++(60:1);
	\draw (\i+1/2,0.866*\j) -- ++(-120:1) -- ++(0:1) -- ++(120:1);
	\draw (\i+1/2,0.866*\j) -- ++(-60:1) -- ++(60:1) -- ++(180:1);
	}
	\else{
	\draw (\i,0.866*\j) -- ++(0:1) -- ++(120:1) -- ++(-120:1);
	\draw (\i,0.866*\j) -- ++(60:1) -- ++(180:1) -- ++(-60:1);
	\draw (\i,0.866*\j) -- ++(120:1) -- ++(-120:1) -- ++(0:1);
	\draw (\i,0.866*\j) -- ++(180:1) -- ++(-60:1) -- ++(60:1);
	\draw (\i,0.866*\j) -- ++(-120:1) -- ++(0:1) -- ++(120:1);
	\draw (\i,0.866*\j) -- ++(-60:1) -- ++(60:1) -- ++(180:1);
	}\fi
}
\filldraw (0,0) circle(2pt) node[below]{$[L_1]$};
\filldraw (1,0) circle(2pt) node[below]{$[L_2]$} ;
\filldraw (1/2,.866) circle(2pt) node[left]{$[L_2']$};
\filldraw (3/2,.866) circle(2pt) node[right]{$[L_3]$};
\filldraw (0,6*.866) circle(2pt) node[above]{$[L_j]$};
\draw [line width=0.5mm, red] (1/2,.866) -- (0,0) -- (1,0); 
\draw [line width=0.5mm, blue] (1,0) -- (2,2*.866) -- (0,6*.866) -- (-1,4*.866) -- (1,0); 
}
\end{tikzpicture}}
\caption{The general configuration for $[L_j]$ in subpath B in the proof of Proposition \ref{prop:elbow move}. \label{fig:region B}}
\end{figure}
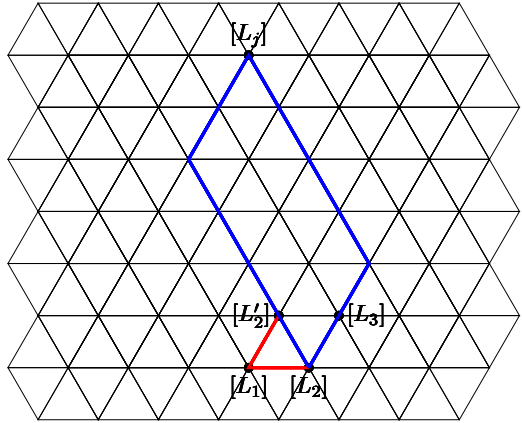

\noindent\underline{\bf Subpath C.}

\noindent{\bf Local Condition:}
Consider region $C$ consisting of the vertices $[L_j]$ such that $b\leq j<c$, in which case $\dif(\gamma_{1,j},\gamma_{2,j})=\{2\}$ and $\dif(\gamma_{2,j},\gamma_{3,j})=\{1,3\}$. In particular, $\gamma_{1,j}=\gamma_{3,j}+(1,1,1)$. These differences imply that $d([L_1],[L_j])=d([L_3],[L_j])$ and 
\begin{align*}
d([L_2],[L_j])&=d([L_1],[L_j])-\omega_2+\omega_1\\
&=d([L_3],[L_j])-\omega_2+\omega_1.
\end{align*}
Let $[L_j]$ be any vertex for $b\leq j<c$. In this case, it may not be possible to place $[L_j]$ in a common apartment with all three of $[L_1], [L_2], [L_3]$ simultaneously. However, this is possible when $d([L_2],[L_j])=k\omega_1$ for some $k$, as is the case when $i=b$ or $i=c-1$.

Consider first $i=b$, so that $d([L_2],[L_b])=k\omega_1$ for some $k$, and $d([L_1],[L_b])=d([L_3],[L_b])=(k-1)\omega_1+\omega_2$. Let $K$ be the vertex adjacent to $[L_2]$ on the straight-path geodesic from $[L_2]$ to $[L_b]$. Since $\lvert \gamma_{1,b}\rvert=\lvert \gamma_{2,b}\rvert=\lvert \gamma_{3,b}\rvert$, and $[L_1]$ and $[L_3]$ are adjacent to $[L_2]$, we have that $\left([L_1],[L_2],K\right)$ and $\left([L_3],[L_2],K\right)$ are a pair of $2$-simplices in the building. The two triangles can be put in a common apartment, along with the geodesic from $[L_2]$ to $[L_b]$. This implies that $K=[L_2']$ and $d([L_2'],[L_b])=(k-1)\omega_1$. See Figure \ref{fig:region C, index b}. The same setup occurs when $i=c-1$. 
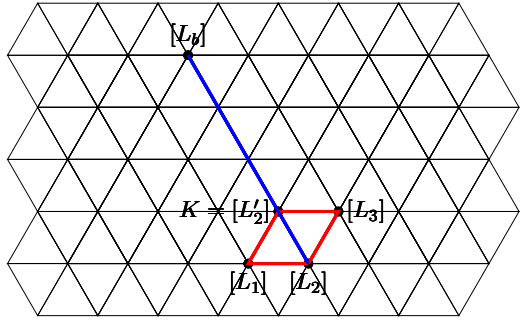
\begin{figure}
\center
\scalemath{.8}{
\begin{tikzpicture}
\foreach \j in {0,...,4}{
\foreach \i in {-3,...,3}{
	\ifodd\j{
	\draw (\i+1/2,0.866*\j) -- ++(0:1) -- ++(120:1) -- ++(-120:1);
	\draw (\i+1/2,0.866*\j) -- ++(60:1) -- ++(180:1) -- ++(-60:1);
	\draw (\i+1/2,0.866*\j) -- ++(120:1) -- ++(-120:1) -- ++(0:1);
	\draw (\i+1/2,0.866*\j) -- ++(180:1) -- ++(-60:1) -- ++(60:1);
	\draw (\i+1/2,0.866*\j) -- ++(-120:1) -- ++(0:1) -- ++(120:1);
	\draw (\i+1/2,0.866*\j) -- ++(-60:1) -- ++(60:1) -- ++(180:1);
	}
	\else{
	\draw (\i,0.866*\j) -- ++(0:1) -- ++(120:1) -- ++(-120:1);
	\draw (\i,0.866*\j) -- ++(60:1) -- ++(180:1) -- ++(-60:1);
	\draw (\i,0.866*\j) -- ++(120:1) -- ++(-120:1) -- ++(0:1);
	\draw (\i,0.866*\j) -- ++(180:1) -- ++(-60:1) -- ++(60:1);
	\draw (\i,0.866*\j) -- ++(-120:1) -- ++(0:1) -- ++(120:1);
	\draw (\i,0.866*\j) -- ++(-60:1) -- ++(60:1) -- ++(180:1);
	}\fi
}
\filldraw (0,0) circle(2pt) node[below]{$[L_1]$};
\filldraw (1,0) circle(2pt) node[below]{$[L_2]$} ;
\filldraw (1/2,.866) circle(2pt) node[left]{$K=[L_2']$};
\filldraw (3/2,.866) circle(2pt) node[right]{$[L_3]$};
\filldraw (-1,4*.866) circle(2pt) node[above]{$[L_b]$};
\draw [line width=0.5mm, red] (1/2,.866) -- (0,0) -- (1,0); 
\draw [line width=0.5mm, red] (1,0) -- (3/2,.866) -- (1/2,.866); 
\draw [line width=0.5mm, blue] (1,0) -- (-1,4*.866); 
}
\end{tikzpicture}}
\caption{The straight-path geodesic from $[L_2]$ to $[L_b]$. \label{fig:region C, index b}}
\end{figure}
\begin{figure}
\scalemath{.68}{
\begin{tikzpicture}
\foreach \j in {0,...,6}{
\foreach \i in {-3,...,3}{
	\ifodd\j{
	\draw (\i+1/2,0.866*\j) -- ++(0:1) -- ++(120:1) -- ++(-120:1);
	\draw (\i+1/2,0.866*\j) -- ++(60:1) -- ++(180:1) -- ++(-60:1);
	\draw (\i+1/2,0.866*\j) -- ++(120:1) -- ++(-120:1) -- ++(0:1);
	\draw (\i+1/2,0.866*\j) -- ++(180:1) -- ++(-60:1) -- ++(60:1);
	\draw (\i+1/2,0.866*\j) -- ++(-120:1) -- ++(0:1) -- ++(120:1);
	\draw (\i+1/2,0.866*\j) -- ++(-60:1) -- ++(60:1) -- ++(180:1);
	}
	\else{
	\draw (\i,0.866*\j) -- ++(0:1) -- ++(120:1) -- ++(-120:1);
	\draw (\i,0.866*\j) -- ++(60:1) -- ++(180:1) -- ++(-60:1);
	\draw (\i,0.866*\j) -- ++(120:1) -- ++(-120:1) -- ++(0:1);
	\draw (\i,0.866*\j) -- ++(180:1) -- ++(-60:1) -- ++(60:1);
	\draw (\i,0.866*\j) -- ++(-120:1) -- ++(0:1) -- ++(120:1);
	\draw (\i,0.866*\j) -- ++(-60:1) -- ++(60:1) -- ++(180:1);
	}\fi
}
\filldraw[gray] (0,0) circle(2pt) node[below]{$[L_1]$};
\filldraw (1,0) circle(2pt) node[below]{$[L_2]$} ;
\filldraw (1/2,.866) circle(2pt) node[below]{$K$};
\filldraw (3/2,.866) circle(2pt) node[right]{$[L_3]$};
\filldraw (-7/2,0.866*5) circle(2pt) node[left]{$[L_j]$};
\draw [line width=0.5mm, red] (1/2,.866) -- (1,0); 
\draw [line width=0.5mm, red] (1,0) -- (3/2,.866) -- (1/2,.866); 
\draw [line width=0.5mm, blue] (3/2,0.866) -- (-3/2,0.866) -- (-7/2,0.866*5) -- (-1/2,0.866*5) -- (3/2,0.866); 
}
\end{tikzpicture}}
\scalemath{.68}{
\begin{tikzpicture}
\foreach \j in {0,...,6}{
\foreach \i in {-3,...,3}{
	\ifodd\j{
	\draw (\i+1/2,0.866*\j) -- ++(0:1) -- ++(120:1) -- ++(-120:1);
	\draw (\i+1/2,0.866*\j) -- ++(60:1) -- ++(180:1) -- ++(-60:1);
	\draw (\i+1/2,0.866*\j) -- ++(120:1) -- ++(-120:1) -- ++(0:1);
	\draw (\i+1/2,0.866*\j) -- ++(180:1) -- ++(-60:1) -- ++(60:1);
	\draw (\i+1/2,0.866*\j) -- ++(-120:1) -- ++(0:1) -- ++(120:1);
	\draw (\i+1/2,0.866*\j) -- ++(-60:1) -- ++(60:1) -- ++(180:1);
	}
	\else{
	\draw (\i,0.866*\j) -- ++(0:1) -- ++(120:1) -- ++(-120:1);
	\draw (\i,0.866*\j) -- ++(60:1) -- ++(180:1) -- ++(-60:1);
	\draw (\i,0.866*\j) -- ++(120:1) -- ++(-120:1) -- ++(0:1);
	\draw (\i,0.866*\j) -- ++(180:1) -- ++(-60:1) -- ++(60:1);
	\draw (\i,0.866*\j) -- ++(-120:1) -- ++(0:1) -- ++(120:1);
	\draw (\i,0.866*\j) -- ++(-60:1) -- ++(60:1) -- ++(180:1);
	}\fi
}
\filldraw (0,0) circle(2pt) node[below]{$[L_1]$};
\filldraw (1,0) circle(2pt) node[below]{$[L_2]$} ;
\filldraw (1/2,.866) circle(2pt) node[right]{$K'$};
\filldraw[gray] (3/2,.866) circle(2pt) node[right]{$[L_3]$};
\filldraw (-1/2,7*.866) circle(2pt) node[above]{$[L_j]$};
\draw [line width=0.5mm, red] (1/2,.866) -- (0,0) -- (1,0); 
\draw [line width=0.5mm, red] (1,0) -- (1/2,.866); 
\draw [line width=0.5mm, blue] (0,0) -- (3/2,3*.866) -- (-1/2,7*.866) -- (-2,.866*4) -- (0,0); 
}
\end{tikzpicture}}
\caption{Left: an apartment containing the edge $([L_2],[L_3])$  and $[L_j]$. Right: an apartment containing the edge $([L_1],[L_2])$  and $[L_j]$. \label{fig:apartments}}
\end{figure}
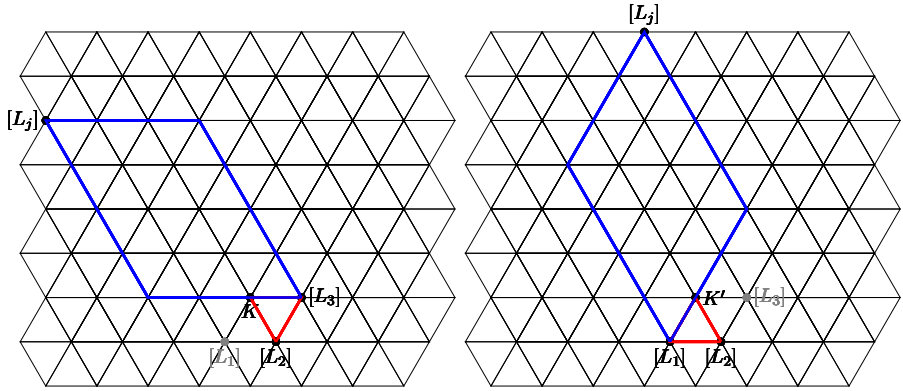
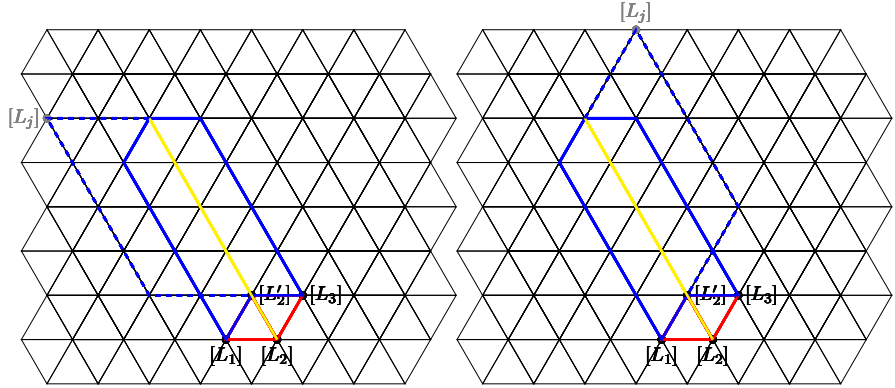
\begin{figure}
\begin{align*}
\scalemath{.68}{
\begin{tikzpicture}
\foreach \j in {0,...,6}{
\foreach \i in {-3,...,3}{
	\ifodd\j{
	\draw (\i+1/2,0.866*\j) -- ++(0:1) -- ++(120:1) -- ++(-120:1);
	\draw (\i+1/2,0.866*\j) -- ++(60:1) -- ++(180:1) -- ++(-60:1);
	\draw (\i+1/2,0.866*\j) -- ++(120:1) -- ++(-120:1) -- ++(0:1);
	\draw (\i+1/2,0.866*\j) -- ++(180:1) -- ++(-60:1) -- ++(60:1);
	\draw (\i+1/2,0.866*\j) -- ++(-120:1) -- ++(0:1) -- ++(120:1);
	\draw (\i+1/2,0.866*\j) -- ++(-60:1) -- ++(60:1) -- ++(180:1);
	}
	\else{
	\draw (\i,0.866*\j) -- ++(0:1) -- ++(120:1) -- ++(-120:1);
	\draw (\i,0.866*\j) -- ++(60:1) -- ++(180:1) -- ++(-60:1);
	\draw (\i,0.866*\j) -- ++(120:1) -- ++(-120:1) -- ++(0:1);
	\draw (\i,0.866*\j) -- ++(180:1) -- ++(-60:1) -- ++(60:1);
	\draw (\i,0.866*\j) -- ++(-120:1) -- ++(0:1) -- ++(120:1);
	\draw (\i,0.866*\j) -- ++(-60:1) -- ++(60:1) -- ++(180:1);
	}\fi
}
\filldraw (0,0) circle(2pt) node[below]{$[L_1]$};
\filldraw (1,0) circle(2pt) node[below]{$[L_2]$} ;
\filldraw (1/2,.866) circle(2pt) node[right]{$[L_2']$};
\filldraw (3/2,.866) circle(2pt) node[right]{$[L_3]$};
\filldraw[gray] (-7/2,0.866*5) circle(2pt) node[left]{$[L_j]$};
\draw [line width=0.5mm, red] (1/2,.866) -- (0,0) -- (1,0); 
\draw [line width=0.5mm, red] (1,0) -- (1/2,.866); 
\draw [line width=0.5mm, red] (1/2,.866) -- (1,0); 
\draw [line width=0.5mm, red] (1,0) -- (3/2,.866) -- (1/2,.866); 
\draw [line width=0.5mm, blue] (-3/2,5*.866) -- (-1/2, 5*.866) -- (3/2,.866) -- (1/2,.866); 
\draw [line width=0.5mm, blue] (-3/2,5*.866) -- (-2, 4*.866)--(0,0) -- (1/2,.866); 
\draw[line width=0.5mm, yellow] (1,0) -- (-3/2,5*.866);
\draw[dashed, line width=.5mm, blue] (-3/2,5*.866) -- (-7/2,5*0.866) -- (-3/2,0.866) -- (1/2,0.866);
}
\end{tikzpicture}}
\scalemath{.68}{
\begin{tikzpicture}
\foreach \j in {0,...,6}{
\foreach \i in {-3,...,3}{
	\ifodd\j{
	\draw (\i+1/2,0.866*\j) -- ++(0:1) -- ++(120:1) -- ++(-120:1);
	\draw (\i+1/2,0.866*\j) -- ++(60:1) -- ++(180:1) -- ++(-60:1);
	\draw (\i+1/2,0.866*\j) -- ++(120:1) -- ++(-120:1) -- ++(0:1);
	\draw (\i+1/2,0.866*\j) -- ++(180:1) -- ++(-60:1) -- ++(60:1);
	\draw (\i+1/2,0.866*\j) -- ++(-120:1) -- ++(0:1) -- ++(120:1);
	\draw (\i+1/2,0.866*\j) -- ++(-60:1) -- ++(60:1) -- ++(180:1);
	}
	\else{
	\draw (\i,0.866*\j) -- ++(0:1) -- ++(120:1) -- ++(-120:1);
	\draw (\i,0.866*\j) -- ++(60:1) -- ++(180:1) -- ++(-60:1);
	\draw (\i,0.866*\j) -- ++(120:1) -- ++(-120:1) -- ++(0:1);
	\draw (\i,0.866*\j) -- ++(180:1) -- ++(-60:1) -- ++(60:1);
	\draw (\i,0.866*\j) -- ++(-120:1) -- ++(0:1) -- ++(120:1);
	\draw (\i,0.866*\j) -- ++(-60:1) -- ++(60:1) -- ++(180:1);
	}\fi
}
\filldraw (0,0) circle(2pt) node[below]{$[L_1]$};
\filldraw (1,0) circle(2pt) node[below]{$[L_2]$} ;
\filldraw (1/2,.866) circle(2pt) node[right]{$[L_2']$};
\filldraw (3/2,.866) circle(2pt) node[right]{$[L_3]$};
\filldraw[gray] (-1/2,7*.866) circle(2pt) node[above]{$[L_j]$};
\draw [line width=0.5mm, red] (1/2,.866) -- (0,0) -- (1,0); 
\draw [line width=0.5mm, red] (1,0) -- (1/2,.866); 
\draw [line width=0.5mm, red] (1/2,.866) -- (1,0); 
\draw [line width=0.5mm, red] (1,0) -- (3/2,.866) -- (1/2,.866); 
\draw [line width=0.5mm, blue] (-3/2,5*.866) -- (-1/2, 5*.866) -- (3/2,.866) -- (1/2,.866); 
\draw [line width=0.5mm, blue] (-3/2,5*.866) -- (-2, 4*.866)--(0,0) -- (1/2,.866); 
\draw[line width=0.5mm, yellow] (1,0) -- (-3/2,5*.866);
\draw[dashed, line width=.5mm, blue] (-3/2,5*.866) -- (-1/2,7*.866) -- (3/2,3*.866) -- (1/2,.866);
}
\end{tikzpicture}}
\end{align*}
\caption{Two views of an apartment containing all four of $[L_1],[L_2],[L_2'],[L_3]$ with $[L_j]$ receding out of the plane of the apartment. \label{fig:flap}}
\end{figure}

Now suppose that $b<j<c-1$. Choose an apartment containing the edge $([L_2],[L_3])$ and $[L_j]$, so that it must also contain the parallelogram of geodesics between $[L_3]$ and $[L_j]$. Since $\gamma_{2,j}=\gamma_{3,j}-\omega_2+\omega_1$, $[L_2]$ is also adjacent to the vertex $K$ on the parallelogram that is an $\omega_2$ step from $[L_3]$ towards $[L_j]$. By choosing an apartment containing the edge $([L_1],[L_2])$ and $[L_j]$, similar reasoning shows that $[L_2]$ is adjacent to the vertex $K'$ on the parallelogram of geodesics from $[L_1]$ to $[L_j]$ that is an $\omega_2$ step from $[L_1]$ towards $[L_j]$. See Figure \ref{fig:apartments}.

But then $K$ and $K'$ are both adjacent to $[L_2]$ in the parallelogram of geodesics from $[L_2]$ to $[L_j]$ and both are an $\omega_1$ step from $[L_2]$ towards $[L_j]$, so $K=~K'$. Since there is a unique vertex adjacent to all three of $[L_1],[L_2],[L_3]$, it follows that $K=[L_2']$. Hence, $[L_2']$ is on a combinatorial geodesic from each of $[L_1],[L_2],[L_3]$ to $[L_j]$, and more specifically $d([L_2'],[L_j])=\gamma_{2,j}-\omega_1=\gamma_{3,j}-\omega_2=\gamma_{1,j}-\omega_2$. See Figure \ref{fig:flap} where all of the solid lines appear in a common apartment, but cannot necessarily be placed in a common apartment with the dashed portion. The dashed portions can be viewed as receding away from the plane (i.e. the yellow line segment can be viewed as a spine with three flaps attached to it). 

Hence, $\dif(\gamma_{1,j},\gamma_{2,j}')=\{1,2\}$ and $\dif(\gamma_{2,j}',\gamma_{3,j})=\{3\}$ for all $b\leq j<c$, so the local condition holds for each of the unit squares across columns with indices in the range $b\leq j<c$. Now consider the unit squares across the columns $b-1$ and $b$. Since $\dif(\gamma_{1,b},\gamma_{2,b}')=\{1,2\}=\dif(\gamma_{1,b-1},\gamma_{2,b-1}')$, the local condition holds for the unit square with weights $\gamma_{1,b-1}$, $\gamma_{1,b}$, $\gamma_{2,b-1}'$, and $\gamma_{2,b}'$. We have that $\gamma_{2,b}=(x,y,y)$ for integers $x>y$, so that $\gamma_{2,b}'=(x-1,y,y)$ and $\gamma_{3,b}=(x-1,y,y-1)$. Recall that $\dif(\gamma_{2,b-1}',\gamma_{3,{b-1}})=\{2\}$. Subtracting~$1$ from the second entry of $\gamma_{2,b}'$ and a sorting yields $\gamma_{3,b}$ as desired.\\

\noindent{\bf Effects on conv:}
Note that for any $i$ such that $b\leq j<c$ either $\conv([L_2],[L_j])$ and $\conv([L_2'],[L_j])$ are both trivial, or $d([L_2],[L_j])=k\omega_1$ and  $d([L_2'],[L_j])=(k-1)\omega_1$ for some $k$, in which case $\conv([L_2],[L_j])=\conv([L_2'],[L_j])\cup \{[L_2]\}$.\\ 

\noindent\underline{\bf Subpath D.} 
Consider the subpath $D$ consisting of vertices $[L_j]$ such that $c\leq j<d$. This region can be empty if $c=d$, which can only happen if $\gamma_{1,c}=(x,x,x)$ for some natural number $x$. If $c<d$, then this case is analogous to the subpath B case with the roles of $[L_1]$ and $[L_3]$ interchanged. We leave the analysis to the reader.\\ 
 \newpage 

\noindent\underline{\bf Subpath E.} 

\noindent{\bf Local Condition:}
Consider the final subpath E that begins with partitions $\gamma_{1,d}=(x,x,y), \gamma_{2,d}=(x,x,y-1), \gamma_{3,d}=(x,x-1,y-1)$ for $x\geq y$. Then $d([L_1],[L_d])=k\omega_2$, $d([L_2],[L_d])=(k+1)\omega_2$, and $d([L_3],[L_d])=(k+1)\omega_2+\omega_1$ for some $k$. For any $j$ such that $d\leq j$, we have $d([L_2],[L_j])=d([L_1],[L_j])+\omega_2$ and $d([L_3],[L_j])=d([L_2],[L_j])+\omega_1$. In this case, all five vertices $[L_1],[L_2],[L_2'],[L_3]$ and $[L_j]$ can be placed in a common apartment giving the configuration on the right of Figure \ref{fig:regions D and E}.

Then replacing $[L_2]$ by $[L_2']$ gives the distance $d([L_2'],[L_j])=d([L_2],[L_j])-\omega_2+\omega_1$, which implies that $\dif(\gamma_{1,j},\gamma_{2,j}')=\{2,3\}$ and $\dif(\gamma_{2,j}',\gamma_{3,j})=\{3\}$. The local rule is satisfied across each unit square, including at index $i=d$.\\

\noindent{\bf Effects on conv:}
Note that $\conv([L_2],[L_j])$ is nontrivial only for indices $j$ such that $\gamma_{2,j}=(k+1)\omega_2$ and $\gamma_{1,j}=k\omega_2$ for some $k$, in which case $\conv([L_2],[L_j])=\conv([L_1],[L_j])\cup\{[L_2]\}$. Similarly, $\conv([L_2'],[L_j])$ is nontrivial only when $d([L_2'],[L_j])=k\omega_1$ and $d([L_1],[L_j])=(k-1)\omega_1$ for some $k$, in which case $\conv([L_2'],[L_j])=\conv([L_1],[L_j])\cup\{[L_2']\}$. \qedhere

\begin{figure}
\begin{align*}
\sloppy
\scalemath{.67}{
\begin{tikzpicture}
\foreach \j in {0,...,4}{
\foreach \i in {-3,...,3}{
	\ifodd\j{
	\draw (\i+1/2,0.866*\j) -- ++(0:1) -- ++(120:1) -- ++(-120:1);
	\draw (\i+1/2,0.866*\j) -- ++(60:1) -- ++(180:1) -- ++(-60:1);
	\draw (\i+1/2,0.866*\j) -- ++(120:1) -- ++(-120:1) -- ++(0:1);
	\draw (\i+1/2,0.866*\j) -- ++(180:1) -- ++(-60:1) -- ++(60:1);
	\draw (\i+1/2,0.866*\j) -- ++(-120:1) -- ++(0:1) -- ++(120:1);
	\draw (\i+1/2,0.866*\j) -- ++(-60:1) -- ++(60:1) -- ++(180:1);
	}
	\else{
	\draw (\i,0.866*\j) -- ++(0:1) -- ++(120:1) -- ++(-120:1);
	\draw (\i,0.866*\j) -- ++(60:1) -- ++(180:1) -- ++(-60:1);
	\draw (\i,0.866*\j) -- ++(120:1) -- ++(-120:1) -- ++(0:1);
	\draw (\i,0.866*\j) -- ++(180:1) -- ++(-60:1) -- ++(60:1);
	\draw (\i,0.866*\j) -- ++(-120:1) -- ++(0:1) -- ++(120:1);
	\draw (\i,0.866*\j) -- ++(-60:1) -- ++(60:1) -- ++(180:1);
	}\fi
}
\filldraw (0,0) circle(2pt) node[below]{$[L_1]$};
\filldraw (1,0) circle(2pt) node[below]{$[L_2]$} ;
\filldraw (1/2,.866) circle(2pt) node[left]{$[L_2']$};
\filldraw (3/2,.866) circle(2pt) node[right]{$[L_3]$};
\filldraw (-3,4*.866) circle(2pt) node[above]{$[L_j]$};
\draw [line width=0.5mm, red] (1,0) -- (3/2,.866) -- (1/2,.866); 
\draw [line width=0.5mm, blue] (1,0) -- (-1,4*.866) -- (-3,4*.866) -- (-1,0) -- (1,0); 
}
\end{tikzpicture}}
\scalemath{.67}{
\begin{tikzpicture}
\foreach \j in {-2,...,2}{
\foreach \i in {-3,...,3}{
	\ifodd\j{
	\draw (\i+1/2,0.866*\j) -- ++(0:1) -- ++(120:1) -- ++(-120:1);
	\draw (\i+1/2,0.866*\j) -- ++(60:1) -- ++(180:1) -- ++(-60:1);
	\draw (\i+1/2,0.866*\j) -- ++(120:1) -- ++(-120:1) -- ++(0:1);
	\draw (\i+1/2,0.866*\j) -- ++(180:1) -- ++(-60:1) -- ++(60:1);
	\draw (\i+1/2,0.866*\j) -- ++(-120:1) -- ++(0:1) -- ++(120:1);
	\draw (\i+1/2,0.866*\j) -- ++(-60:1) -- ++(60:1) -- ++(180:1);
	}
	\else{
	\draw (\i,0.866*\j) -- ++(0:1) -- ++(120:1) -- ++(-120:1);
	\draw (\i,0.866*\j) -- ++(60:1) -- ++(180:1) -- ++(-60:1);
	\draw (\i,0.866*\j) -- ++(120:1) -- ++(-120:1) -- ++(0:1);
	\draw (\i,0.866*\j) -- ++(180:1) -- ++(-60:1) -- ++(60:1);
	\draw (\i,0.866*\j) -- ++(-120:1) -- ++(0:1) -- ++(120:1);
	\draw (\i,0.866*\j) -- ++(-60:1) -- ++(60:1) -- ++(180:1);
	}\fi
}
\filldraw (0,0) circle(2pt) node[below]{$[L_1]$};
\filldraw (1,0) circle(2pt) node[below]{$[L_2]$} ;
\filldraw (1/2,.866) circle(2pt) node[left]{$[L_2']$};
\filldraw (3/2,.866) circle(2pt) node[right]{$[L_3]$};
\filldraw (-4,-2*.866) circle(2pt) node[left]{$[L_j]$};
\draw [line width=0.5mm, red] (1,0) -- (3/2,.866) -- (1/2,.866) -- (1,0); 
\draw [line width=0.5mm, blue] (1,0) -- (-3,0) -- (-4,-2*.866) -- (0,-2*.866) -- (1,0); 
}
\end{tikzpicture}}
\end{align*}
\caption{Left: the configuration for subpath D. Right: the configuration for subpath E.\label{fig:regions D and E}}
\end{figure} \end{proof}

\begin{Thm}
Let $P$ be a generic polygon. Then $\conv(P)$ is a CAT(0) triangulation of $P$ in the affine building.
\end{Thm}
\begin{proof}
Induct on $n$, the number of vertices in $P$. If $P$ contains a $U$-turn at vertex $i$, then the polygon $P'$ with vertex $i$ removed and the $i-1$, $i+1$ vertices identified is generic by Proposition \ref{prop:U-turn}. By induction $\conv(P')$ is a CAT(0) triangulation of $P'$ in the building. Adding $[L_i]$ back in gives $\conv(P)=\conv(P')\cup\{L_i\}$ by the second part of Proposition \ref{prop:U-turn}, and is a CAT(0) triangulation of $P$.  If $P$ contains a sharp corner at vertex $i$, then similarly apply Proposition \ref{prop:sharp corner} to show that $\conv(P)=\conv(P')\cup\{L_i\}$, which is a CAT(0) triangulation of $P$ in the building. 

Now suppose that $P$ has no U-turns and no sharp corners. By Lemma \ref{lem:double elbow}, $P$ must contain a double-elbow configuration of some length $a$, which we may assume occurs at $[L_1],[L_2],[L_3],\ldots,[L_a]$. Let $[L_2']$ be the unique vertex in the building other than $[L_2]$ that is adjacent to both $[L_1]$ and $[L_3]$. Then by Proposition \ref{prop:elbow move} the polygon with $[L_2]$ replaced by $[L_2']$ is generic, but has the same number of vertices. It now has a double-elbow configuration at $[L_2'],\ldots,[L_a]$. After $a-2$ applications of Proposition \ref{prop:elbow move}, the new generic polygon has a sharp corner at $[L_{a-2}'],[L_{a-1}'],[L_a]$. Let $P'$ be the generic polygon with this sharp corner removed, so that by induction $\conv(P')$ is a CAT(0) triangulation of $P'$. Then $\conv(P)=\conv(P')\cup\{[L_2],[L_3],\ldots,[L_{a-1}]\}$ by repeated application of Proposition \ref{prop:elbow move}.

We claim that the resulting set $\conv(P)$ is a CAT(0) triangulation. If it were not, then an exterior vertex of $\conv(P')$ became an interior vertex when passing to $\conv(P)$, whose link contains fewer than $6$ edges. However, this is impossible in the affine building, hence $\conv(P)$ is a CAT(0) triangulation.
\end{proof}

The proof of the main theorem now follows easily.

\begin{Thm}\label{thm:main}
Let $P_1,\ldots,P_d$ be $d$ generic polygons, one from each of the components of the polygon space $\Poly(\vec\lambda)$. Then the duals of $\conv(P_1),\ldots,\conv(P_d)$ form the non-elliptic web basis for $\Inv(\vec\lambda)$.
\end{Thm}
\begin{proof}
Each $\conv(P_k)$ is a CAT(0) diskoid by the previous theorem, so corresponds to a non-elliptic web. By induction and similar arguments as in the preceding propositions, one can show that $\conv(P_k)$ contains a combinatorial geodesic between any pair of vertices. Since the $P_k$ are contained in different components of $\Poly(\vec\lambda)$, any two $P_k$ differ in at least one distance $d([L_i],[L_j])$ for some $i$, $j$. Hence the $\conv(P_k)$ are distinct.
\end{proof}

\addtocontents{toc}{\protect\setcounter{tocdepth}{2}} 

\bibliographystyle{amsalpha}
\bibliography{biblio}

\end{document}